\documentclass{amsart}
\usepackage{a4wide}
\usepackage{hyperref}
\usepackage{graphicx} 
\usepackage{amsmath,amssymb,amsfonts,amsthm}
\usepackage{mathtools}
\usepackage{tikz}
\usetikzlibrary{arrows.meta}
\usepackage{xcolor}

\definecolor{white}{rgb}{1,1,1}
\definecolor{Black}{rgb}{0,0,0}

\newtheorem{thm}{Theorem}[section]
\newtheorem{defn}[thm]{Definition}
\newtheorem{rem}[thm]{Remark}
\newtheorem{ex}[thm]{Example}
\newtheorem{lem}[thm]{Lemma}
\newtheorem{prop}[thm]{Proposition}

\newtheorem{cor}[thm]{Corollary}

\newtheorem{prob}[thm]{Problem}

\newcommand{\Ino}{\textup{I}}
\newcommand{\IIno}{\textup{I\!I}}
\newcommand{\IIIno}{\textup{I\!I\!I}}
\newcommand{\IVno}{\textup{I\!V}}
\newcommand{\Vno}{\textup{V}}
\newcommand{\prs}{\operatorname{pr}_s}
\newcommand{\RI}{\mathcal{R}\I}
\newcommand{\RII}{\mathcal{R}\II}
\newcommand{\RIII}{\mathcal{R}\III}
\newcommand{\RIV}{\mathcal{R}\IV}
\newcommand{\RIVO}{\mathcal{R}\IVO}
\newcommand{\RIVU}{\mathcal{R}\IVU}
\newcommand{\RV}{\mathcal{R}\V}

\newcommand{\MGSordi}{\mathrm{MGS}_{\mathrm{ord}}}
\newcommand{\GSordi}{\mathrm{GS}_{\mathrm{ord}}}
\newcommand{\MGSsingu}{\mathrm{MGS}_{\mathrm{sing}}}
\newcommand{\GSsingu}{\mathrm{GS}_{\mathrm{sing}}}
\newcommand{\am}{\includegraphics[width=0.125\linewidth]{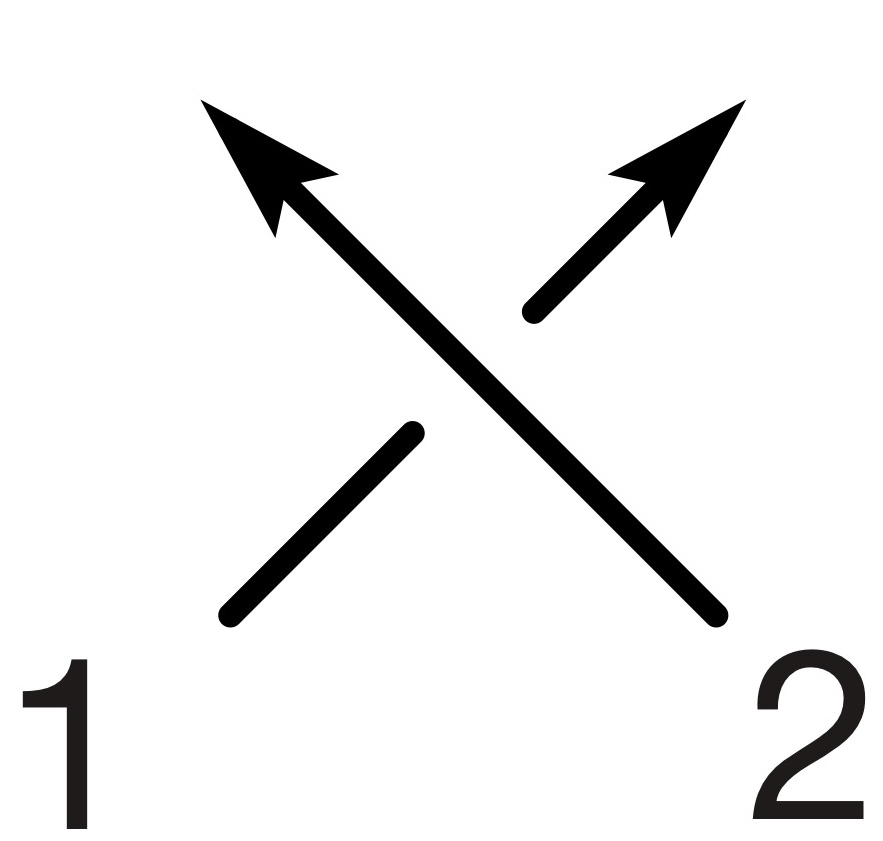}}
\newcommand{\ap}{\includegraphics[width=0.125\linewidth]{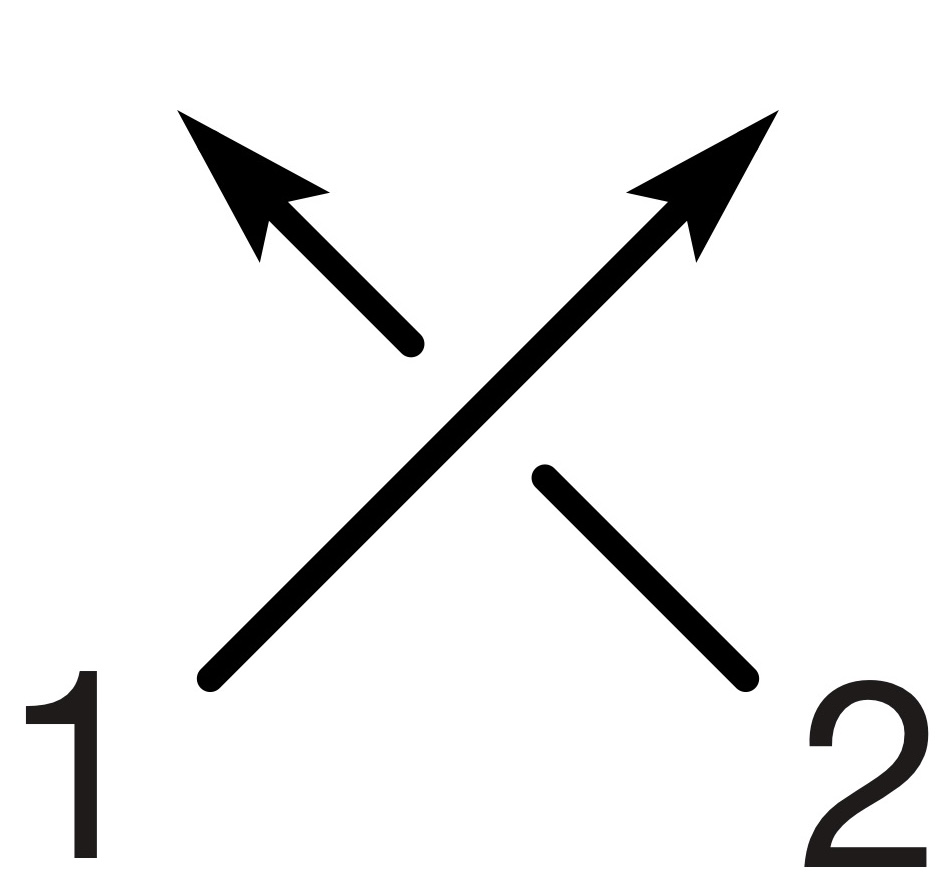}}
\newcommand{\bm}{\includegraphics[width=0.125\linewidth]{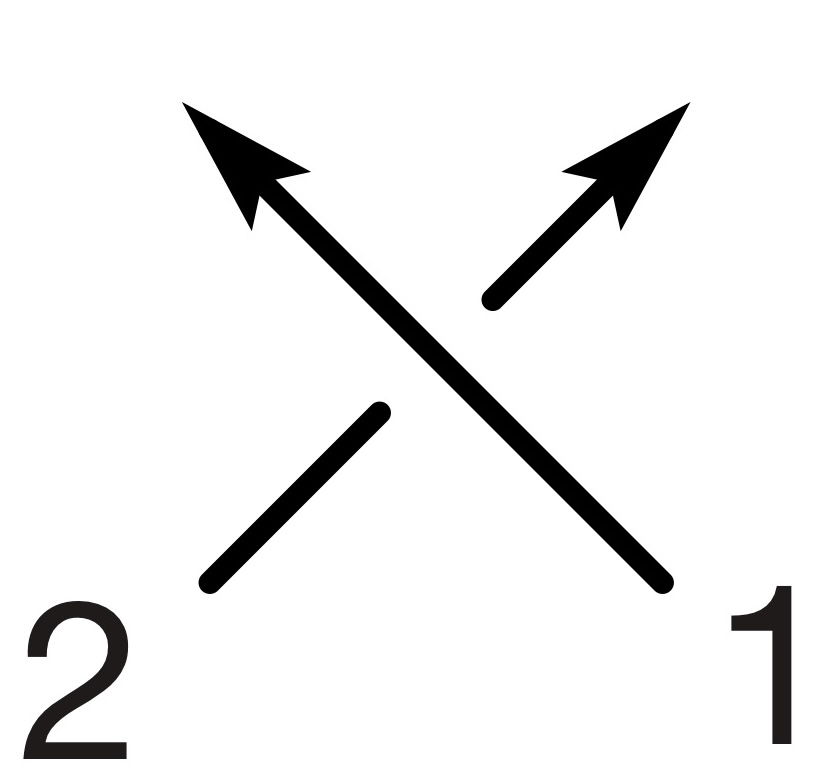}}
\newcommand{\bp}{\includegraphics[width=0.125\linewidth]{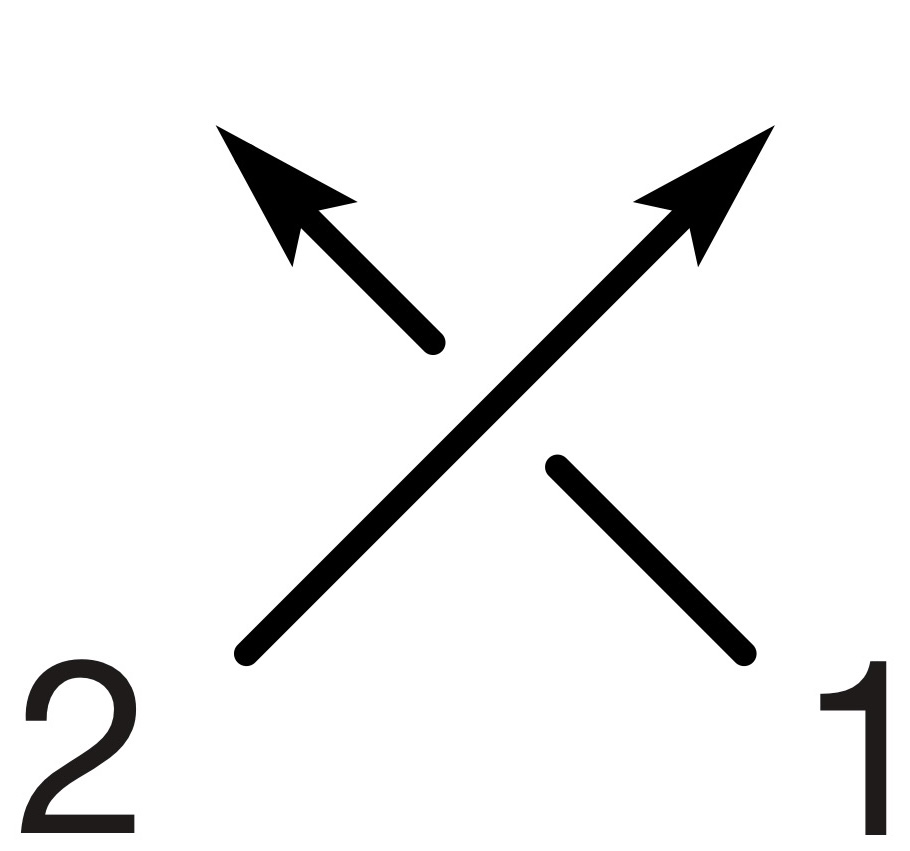}}
\newcommand{\BDPmp}{\rotatebox{180}{\scalebox{0.2}{\lower20ex\hbox{{\input{Reidemeister-move/ADP(-+).tex}}}}}}
\newcommand{\BDPpm}{\rotatebox{180}{\scalebox{0.2}{\lower20ex\hbox{{\input{Reidemeister-move/ADP(+-).tex}}}}}}
\newcommand{\BDCpm}{\rotatebox{180}{\scalebox{0.2}{\lower20ex\hbox{{\input{Reidemeister-move/ADC(+-).tex}}}}}}
\newcommand{\BDCmp}{\rotatebox{180}{\scalebox{0.2}{\lower20ex\hbox{{\input{Reidemeister-move/ADC(-+).tex}}}}}}

\newcommand{\R}{\mathcal{R}}
\newcommand{\SR}{\mathcal{SR}} 
\newcommand{\I}{\textup{I}~}
\newcommand{\II}{\textup{I\!I}~}
\newcommand{\III}{\textup{I\!I\!I}~}
\newcommand{\IV}{\textup{I\!V}~}
\newcommand{\IVO}{\textup{I\!V}_{\!O}}
\newcommand{\IVU}{\textup{I\!V}_{\!U}}
\newcommand{\V}{\textup{V}~}
\newcommand{\tI}{\textup{type~I}}
\newcommand{\tII}{\textup{type~I\!I}}
\newcommand{\tIII}{\textup{type~I\!I\!I}}
\newcommand{\tIV}{\textup{type~I\!V}}
\newcommand{\tV}{\textup{type~V}}
\newcommand{\oneA}{\textup{1a}}
\newcommand{\oneB}{\textup{1b}}
\newcommand{\oneC}{\textup{1c}}
\newcommand{\oneD}{\textup{1d}}

\newcommand{\twoA}{\textup{2a}}
\newcommand{\twoB}{\textup{2b}}
\newcommand{\twoC}{\textup{2c}}
\newcommand{\twoD}{\textup{2d}}
\newcommand{\threeA}{\textup{3a}}
\newcommand{\threeB}{\textup{3b}}
\newcommand{\threeC}{\textup{3c}}
\newcommand{\threeD}{\textup{3d}}
\newcommand{\threeE}{\textup{3e}}
\newcommand{\threeF}{\textup{3f}}
\newcommand{\threeG}{\textup{3g}}
\newcommand{\threeH}{\textup{3h}}

\newcommand{\fourA}{\textup{4a}}
\newcommand{\fourB}{\textup{4b}}
\newcommand{\fourC}{\textup{4c}}
\newcommand{\fourD}{\textup{4d}}
\newcommand{\fourE}{\textup{4e}}
\newcommand{\fourF}{\textup{4f}}
\newcommand{\fourG}{\textup{4g}}
\newcommand{\fourH}{\textup{4h}}
\newcommand{\fourAS}{\operatorname{4\ast}}
\newcommand{\fourO}{\operatorname{4\ast}}
\newcommand{\fourU}{\operatorname{4\sharp}}
\newcommand{\fiveA}{\textup{5a}}
\newcommand{\fiveB}{\textup{5b}}
\newcommand{\fiveC}{\textup{5c}}
\newcommand{\fiveD}{\textup{5d}}
\newcommand{\fiveE}{\textup{5e}}
\newcommand{\fiveF}{\textup{5f}}
\newcommand{\fiveST}{\operatorname{5\star}}

\newcommand{\Ia}  {\scalebox{0.4}{\lower14ex\hbox{{\input{Reidemeister-move/1a.tex   }}}}}
\newcommand{\Ib}  {\scalebox{0.4}{\lower14ex\hbox{{\input{Reidemeister-move/1b.tex   }}}}}
\newcommand{\Ic}  {\scalebox{0.4}{\lower14ex\hbox{{\input{Reidemeister-move/1c.tex   }}}}}
\newcommand{\Id}  {\scalebox{0.4}{\lower14ex\hbox{{\input{Reidemeister-move/1d.tex   }}}}}
\newcommand{\Iu}  {\scalebox{0.4}{\lower14ex\hbox{{\input{Reidemeister-move/1ue.tex  }}}}}
\newcommand{\Is}  {\scalebox{0.4}{\lower14ex\hbox{{\input{Reidemeister-move/1sita.tex}}}}}
\newcommand{\Io}  {\scalebox{0.4}{\lower14ex\hbox{{\input{Reidemeister-move/1o.tex   }}}}}
\newcommand{\Iun} {\scalebox{0.4}{\lower14ex\hbox{{\input{Reidemeister-move/1u.tex   }}}}}
\newcommand{\Imid}{\scalebox{0.4}{\lower14ex\hbox{{\input{Reidemeister-move/tyokusenn.tex   }}}}}
\newcommand{\IIa}  {\scalebox{0.4}{\lower14ex\hbox{{\input{Reidemeister-move/2a.tex   }}}}}
\newcommand{\IIb}  {\scalebox{0.4}{\lower14ex\hbox{{\input{Reidemeister-move/2b.tex   }}}}}
\newcommand{\IIc}  {\scalebox{0.4}{\lower14ex\hbox{{\input{Reidemeister-move/2c.tex   }}}}}
\newcommand{\IId}  {\scalebox{0.4}{\lower14ex\hbox{{\input{Reidemeister-move/2d.tex   }}}}}
\newcommand{\IIuu}  {\scalebox{0.4}{\lower14ex\hbox{{\input{Reidemeister-move/2uu.tex   }}}}}
\newcommand{\IIus}  {\scalebox{0.4}{\lower14ex\hbox{{\input{Reidemeister-move/2us.tex   }}}}}
\newcommand{\IIsu}  {\scalebox{0.4}{\lower14ex\hbox{{\input{Reidemeister-move/2su.tex   }}}}}
\newcommand{\IIss}  {\scalebox{0.4}{\lower14ex\hbox{{\input{Reidemeister-move/2ss.tex   }}}}}
\newcommand{\IIo} {\scalebox{0.4}{\lower14ex\hbox{{\input{Reidemeister-move/2o.tex   }}}}}
\newcommand{\IIu} {\scalebox{0.4}{\lower14ex\hbox{{\input{Reidemeister-move/2u.tex   }}}}}
\newcommand{\IIn} {\scalebox{0.4}{\lower14ex\hbox{{\input{Reidemeister-move/2n.tex   }}}}}
\newcommand{\IIIo}{\scalebox{0.4}{\lower14ex\hbox{{\input{Reidemeister-move/3o.tex   }}}}}
\newcommand{\IIIu}{\scalebox{0.4}{\lower14ex\hbox{{\input{Reidemeister-move/3u.tex   }}}}}
\newcommand{\IIIocc}{\scalebox{0.4}{\lower14ex\hbox{{\input{Reidemeister-move/3o-cc.tex   }}}}}
\newcommand{\IIIucc}{\scalebox{0.4}{\lower14ex\hbox{{\input{Reidemeister-move/3u-cc.tex   }}}}}
\newcommand{\IIIao}  {\scalebox{0.5}{\lower14ex\hbox{{\input{Reidemeister-move/3a-o.tex   }}}}}
\newcommand{\IIIau}  {\scalebox{0.5}{\lower14ex\hbox{{\input{Reidemeister-move/3a-u.tex   }}}}}
\newcommand{\IIIbo}  {\scalebox{0.5}{\lower14ex\hbox{{\input{Reidemeister-move/3b-o.tex   }}}}}
\newcommand{\IIIbu}  {\scalebox{0.5}{\lower14ex\hbox{{\input{Reidemeister-move/3b-u.tex   }}}}}
\newcommand{\IIIco}  {\scalebox{0.5}{\lower14ex\hbox{{\input{Reidemeister-move/3c-o.tex   }}}}}
\newcommand{\IIIcu}  {\scalebox{0.5}{\lower14ex\hbox{{\input{Reidemeister-move/3c-u.tex   }}}}}
\newcommand{\IIIdo}  {\scalebox{0.5}{\lower14ex\hbox{{\input{Reidemeister-move/3d-o.tex   }}}}}
\newcommand{\IIIdu}  {\scalebox{0.5}{\lower14ex\hbox{{\input{Reidemeister-move/3d-u.tex   }}}}}
\newcommand{\IIIeo}  {\scalebox{0.5}{\lower14ex\hbox{{\input{Reidemeister-move/3e-o.tex   }}}}}
\newcommand{\IIIeu}  {\scalebox{0.5}{\lower14ex\hbox{{\input{Reidemeister-move/3e-u.tex   }}}}}
\newcommand{\IIIfo}  {\scalebox{0.5}{\lower14ex\hbox{{\input{Reidemeister-move/3f-o.tex   }}}}}
\newcommand{\IIIfu}  {\scalebox{0.5}{\lower14ex\hbox{{\input{Reidemeister-move/3f-u.tex   }}}}}
\newcommand{\IIIgo}  {\scalebox{0.5}{\lower14ex\hbox{{\input{Reidemeister-move/3g-o.tex   }}}}}
\newcommand{\IIIgu}  {\scalebox{0.5}{\lower14ex\hbox{{\input{Reidemeister-move/3g-u.tex   }}}}}
\newcommand{\IIIho}  {\scalebox{0.5}{\lower14ex\hbox{{\input{Reidemeister-move/3h-o.tex   }}}}}
\newcommand{\IIIhu}  {\scalebox{0.5}{\lower14ex\hbox{{\input{Reidemeister-move/3h-u.tex   }}}}}
\newcommand{\IVo}{\scalebox{0.4}{\lower14ex\hbox{{\input{Reidemeister-move/4o.tex   }}}}}
\newcommand{\IVu}{\scalebox{0.4}{\lower14ex\hbox{{\input{Reidemeister-move/4u.tex   }}}}}
\newcommand{\IVoR}{\rotatebox{180}{\scalebox{0.4}{\lower24ex\hbox{{\input{Reidemeister-move/4o.tex   }}}}}}
\newcommand{\IVuR}{\rotatebox{180}{\scalebox{0.4}{\lower24ex\hbox{{\input{Reidemeister-move/4u.tex   }}}}}}
\newcommand{\Vo}{\scalebox{0.4}{\lower14ex\hbox{{\input{Reidemeister-move/5o.tex   }}}}}
\newcommand{\Vu}{\scalebox{0.4}{\lower14ex\hbox{{\input{Reidemeister-move/5u.tex   }}}}}
\newcommand{\VoR}{\rotatebox{180}{\scalebox{0.4}{\lower24ex\hbox{{\input{Reidemeister-move/5o.tex   }}}}}}
\newcommand{\VuR}{\rotatebox{180}{\scalebox{0.4}{\lower24ex\hbox{{\input{Reidemeister-move/5u.tex   }}}}}}

\newcommand{\IVao}  {\scalebox{0.6}{\lower14ex\hbox{{\input{singular-Reidemeister-move/4a-o.tex   }}}}}
\newcommand{\IVau}  {\scalebox{0.6}{\lower14ex\hbox{{\input{singular-Reidemeister-move/4a-u.tex   }}}}}
\newcommand{\IVbo}  {\scalebox{0.6}{\lower14ex\hbox{{\input{singular-Reidemeister-move/4b-o.tex   }}}}}
\newcommand{\IVbu}  {\scalebox{0.6}{\lower14ex\hbox{{\input{singular-Reidemeister-move/4b-u.tex   }}}}}
\newcommand{\IVco}  {\scalebox{0.6}{\lower14ex\hbox{{\input{singular-Reidemeister-move/4c-o.tex   }}}}}
\newcommand{\IVcu}  {\scalebox{0.6}{\lower14ex\hbox{{\input{singular-Reidemeister-move/4c-u.tex   }}}}}
\newcommand{\IVdo}  {\scalebox{0.6}{\lower14ex\hbox{{\input{singular-Reidemeister-move/4d-o.tex   }}}}}
\newcommand{\IVdu}  {\scalebox{0.6}{\lower14ex\hbox{{\input{singular-Reidemeister-move/4d-u.tex   }}}}}

\newcommand{\IVaX}  {\scalebox{0.7}{\lower14ex\hbox{{\input{singular-Reidemeister-move/4a-X.tex   }}}}}

\newcommand{\IVeo}  {\scalebox{0.6}{\lower14ex\hbox{{\input{singular-Reidemeister-move/4e-o.tex   }}}}}
\newcommand{\IVeu}  {\scalebox{0.6}{\lower14ex\hbox{{\input{singular-Reidemeister-move/4e-u.tex   }}}}}
\newcommand{\IVfo}  {\scalebox{0.6}{\lower14ex\hbox{{\input{singular-Reidemeister-move/4f-o.tex   }}}}}
\newcommand{\IVfu}  {\scalebox{0.6}{\lower14ex\hbox{{\input{singular-Reidemeister-move/4f-u.tex   }}}}}
\newcommand{\IVgo}  {\scalebox{0.6}{\lower14ex\hbox{{\input{singular-Reidemeister-move/4g-o.tex   }}}}}
\newcommand{\IVgu}  {\scalebox{0.6}{\lower14ex\hbox{{\input{singular-Reidemeister-move/4g-u.tex   }}}}}
\newcommand{\IVho}  {\scalebox{0.6}{\lower14ex\hbox{{\input{singular-Reidemeister-move/4h-o.tex   }}}}}
\newcommand{\IVhu}  {\scalebox{0.6}{\lower14ex\hbox{{\input{singular-Reidemeister-move/4h-u.tex   }}}}}

\newcommand{\Vao}  {\scalebox{0.5}{\lower14ex\hbox{{\input{singular-Reidemeister-move/5a-o.tex   }}}}}
\newcommand{\Vau}  {\scalebox{0.5}{\lower14ex\hbox{{\input{singular-Reidemeister-move/5a-u.tex   }}}}}
\newcommand{\Vbo}  {\scalebox{0.5}{\lower14ex\hbox{{\input{singular-Reidemeister-move/5b-o.tex   }}}}}
\newcommand{\Vbu}  {\scalebox{0.5}{\lower14ex\hbox{{\input{singular-Reidemeister-move/5b-u.tex   }}}}}
\newcommand{\Vco}  {\scalebox{0.5}{\lower14ex\hbox{{\input{singular-Reidemeister-move/5c-o.tex   }}}}}
\newcommand{\Vcu}  {\scalebox{0.5}{\lower14ex\hbox{{\input{singular-Reidemeister-move/5c-u.tex   }}}}}
\newcommand{\Vdo}  {\scalebox{0.5}{\lower14ex\hbox{{\input{singular-Reidemeister-move/5d-o.tex   }}}}}
\newcommand{\Vdu}  {\scalebox{0.5}{\lower14ex\hbox{{\input{singular-Reidemeister-move/5d-u.tex   }}}}}
\newcommand{\Veo}  {\scalebox{0.5}{\lower14ex\hbox{{\input{singular-Reidemeister-move/5e-o.tex   }}}}}
\newcommand{\Veu}  {\scalebox{0.5}{\lower14ex\hbox{{\input{singular-Reidemeister-move/5e-u.tex   }}}}}
\newcommand{\Vfo}  {\scalebox{0.5}{\lower14ex\hbox{{\input{singular-Reidemeister-move/5f-o.tex   }}}}}
\newcommand{\Vfu}  {\scalebox{0.5}{\lower14ex\hbox{{\input{singular-Reidemeister-move/5f-u.tex   }}}}}

\newcommand{\Isiso}  {\scalebox{0.4}{\lower14ex\hbox{{\input{Reidemeister-move/1sita-iso.tex}}}}}
\newcommand{\IsIIc}  {\scalebox{0.4}{\lower14ex\hbox{{\input{Reidemeister-move/1sita-2c.tex}}}}}
\newcommand{\IsIIcinv}  {\scalebox{0.4}{\lower14ex\hbox{{\input{Reidemeister-move/1sita-2c-inv}}}}}
\newcommand{\Iuiso}  {\scalebox{0.4}{\lower14ex\hbox{{\input{Reidemeister-move/1ue-iso.tex}}}}}
\newcommand{\IuIId}  {\scalebox{0.4}{\lower14ex\hbox{{\input{Reidemeister-move/1ue-2d.tex}}}}}
\newcommand{\IuIIdinv}  {\scalebox{0.4}{\lower14ex\hbox{{\input{Reidemeister-move/1ue-2d-inv.tex}}}}}
\newcommand{\IIuuIaIId}  {\scalebox{0.4}{\lower14ex\hbox{{\input{Reidemeister-move/2uu-1a-2d.tex}}}}}
\newcommand{\IIuuIaIIdIIId}  {\scalebox{0.4}{\lower14ex\hbox{{\input{Reidemeister-move/2uu-1a-2d-3d.tex}}}}}
\newcommand{\IIIaoIIc}  {\scalebox{0.4}{\lower14ex\hbox{{\input{Reidemeister-move/3a-o-2c.tex}}}}}
\newcommand{\IIIaoIIcIIIc}  {\scalebox{0.4}{\lower14ex\hbox{{\input{Reidemeister-move/3a-o-2c-3c.tex}}}}}

\newcommand{\IVaoIIc}  {\scalebox{0.6}{\lower14ex\hbox{{\input{singular-Reidemeister-move/4a-2c}}}}}
\newcommand{\IVaoIIcIVb}  {\scalebox{0.6}{\lower14ex\hbox{{\input{singular-Reidemeister-move/4a-2c-4b}}}}}

\newcommand{\VaoIb}  {\scalebox{0.5}{\lower14ex\hbox{{\input{singular-Reidemeister-move/5a-1b}}}}}
\newcommand{\VaoIbIVb}  {\scalebox{0.5}{\lower14ex\hbox{{\input{singular-Reidemeister-move/5a-1b-4b}}}}}
\newcommand{\VaoIbIVbiso}  {\scalebox{0.5}{\lower14ex\hbox{{\input{singular-Reidemeister-move/5a-1b-4b-iso}}}}}
\newcommand{\VaoIbIVbisoVe}  {\scalebox{0.5}{\lower14ex\hbox{{\input{singular-Reidemeister-move/5a-1b-4b-iso-5e}}}}}
\newcommand{\VaoIbVIbisoVeiso}  {\scalebox{0.5}{\lower14ex\hbox{{\input{singular-Reidemeister-move/5a-1b-4b-iso-5e-iso}}}}}
\newcommand{\VaoIbIVbisoVeisoIVg}  {\scalebox{0.5}{\lower14ex\hbox{{\input{singular-Reidemeister-move/5a-1b-4b-iso-5e-iso-4g}}}}}
\newcommand{\VaoIbIVbisoVeisoIVgIb}  {\scalebox{0.5}{\lower14ex\hbox{{\input{singular-Reidemeister-move/5a-1b-4b-iso-5e-iso-4g-1b}}}}}

\newcommand{\VauL}  {\scalebox{0.5}{\lower14ex\hbox{{\input{singular-Reidemeister-move/5a-u-long}}}}}
\newcommand{\VauLIIaIIb}  {\scalebox{0.5}{\lower14ex\hbox{{\input{singular-Reidemeister-move/5a-u-long-2a2b}}}}}
\newcommand{\VauLIIaIIbVd}  {\scalebox{0.5}{\lower14ex\hbox{{\input{singular-Reidemeister-move/5a-u-long-2a2b-5d}}}}}
\newcommand{\VaoL}  {\scalebox{0.5}{\lower14ex\hbox{{\input{singular-Reidemeister-move/5a-o-long}}}}}

\newcommand{\ADPpm}  {\scalebox{0.2}{\lower14ex\hbox{{\input{Reidemeister-move/ADP(+-).tex}}}}}
\newcommand{\ADPmp}  {\scalebox{0.2}{\lower14ex\hbox{{\input{Reidemeister-move/ADP(-+).tex}}}}}
\newcommand{\ADCpm}  {\scalebox{0.2}{\lower14ex\hbox{{\input{Reidemeister-move/ADC(+-).tex}}}}}
\newcommand{\ADCmp}  {\scalebox{0.2}{\lower14ex\hbox{{\input{Reidemeister-move/ADC(-+).tex}}}}}
\newcommand{\ADC}    {\scalebox{0.2}{\lower14ex\hbox{{\input{Reidemeister-move/ADC.tex}}}}}
\newcommand{\ADP}    {\scalebox{0.2}{\lower14ex\hbox{{\input{Reidemeister-move/ADP.tex}}}}}

\newcommand{\DLone}  {\scalebox{0.6}{\lower14ex\hbox{{\input{Reidemeister-move/DL1.tex}}}}}
\newcommand{\DLtwo}  {\scalebox{0.6}{\lower14ex\hbox{{\input{Reidemeister-move/DL2.tex}}}}}

\newcommand{\GDex}{{\scalebox{0.2}{\lower14ex\hbox{{\input{Reidemeister-move/GD-ex.tex}}}}}}
\newcommand{\Hopf}{{\scalebox{0.45}{\lower14ex\hbox{{\input{Reidemeister-move/Hopf-Link.tex}}}}}}

\newcommand{\SHopfO}{\includegraphics[width=0.475\linewidth]{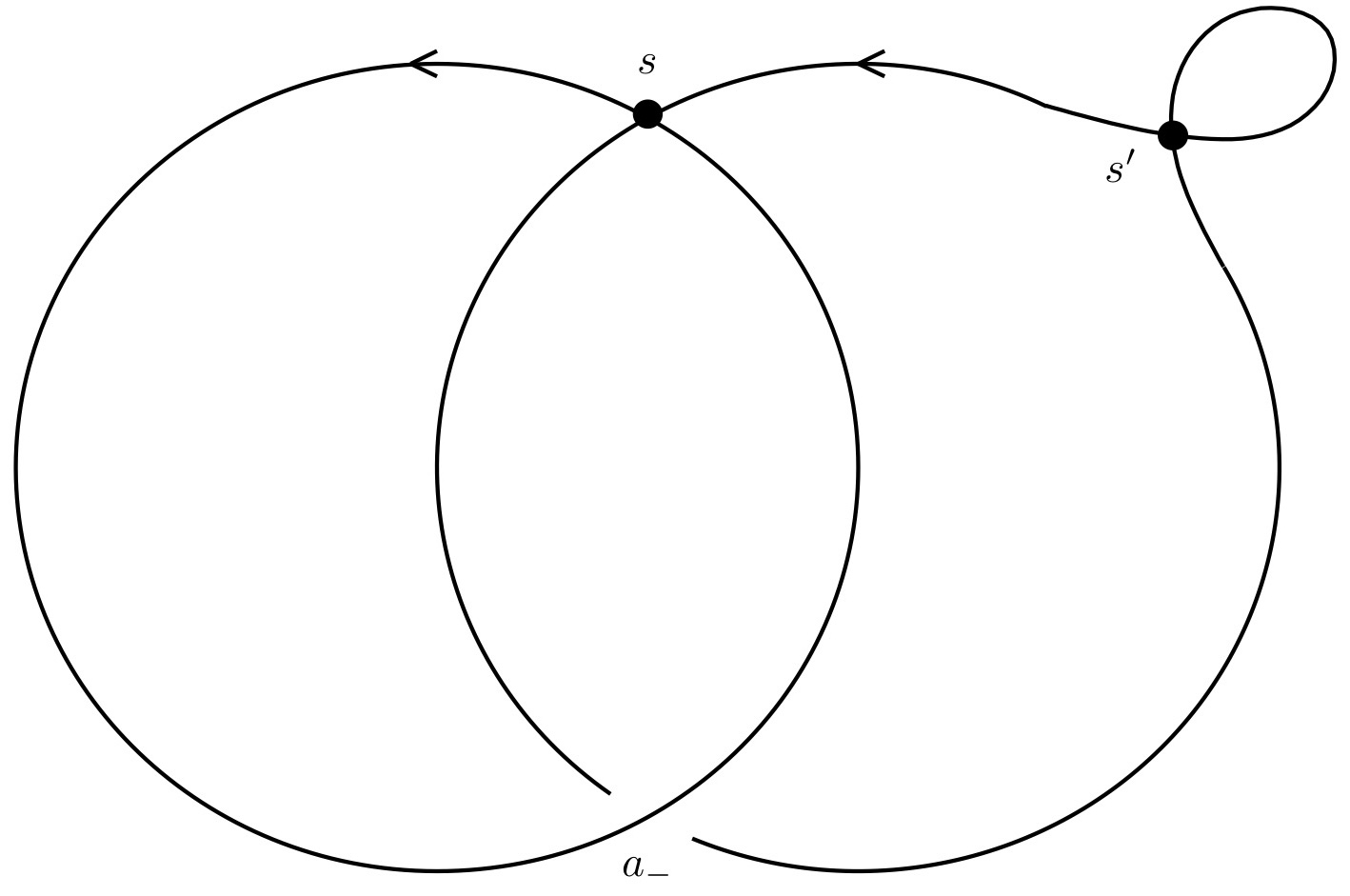}}
\newcommand{\SHopfU}{\includegraphics[width=0.475\linewidth]{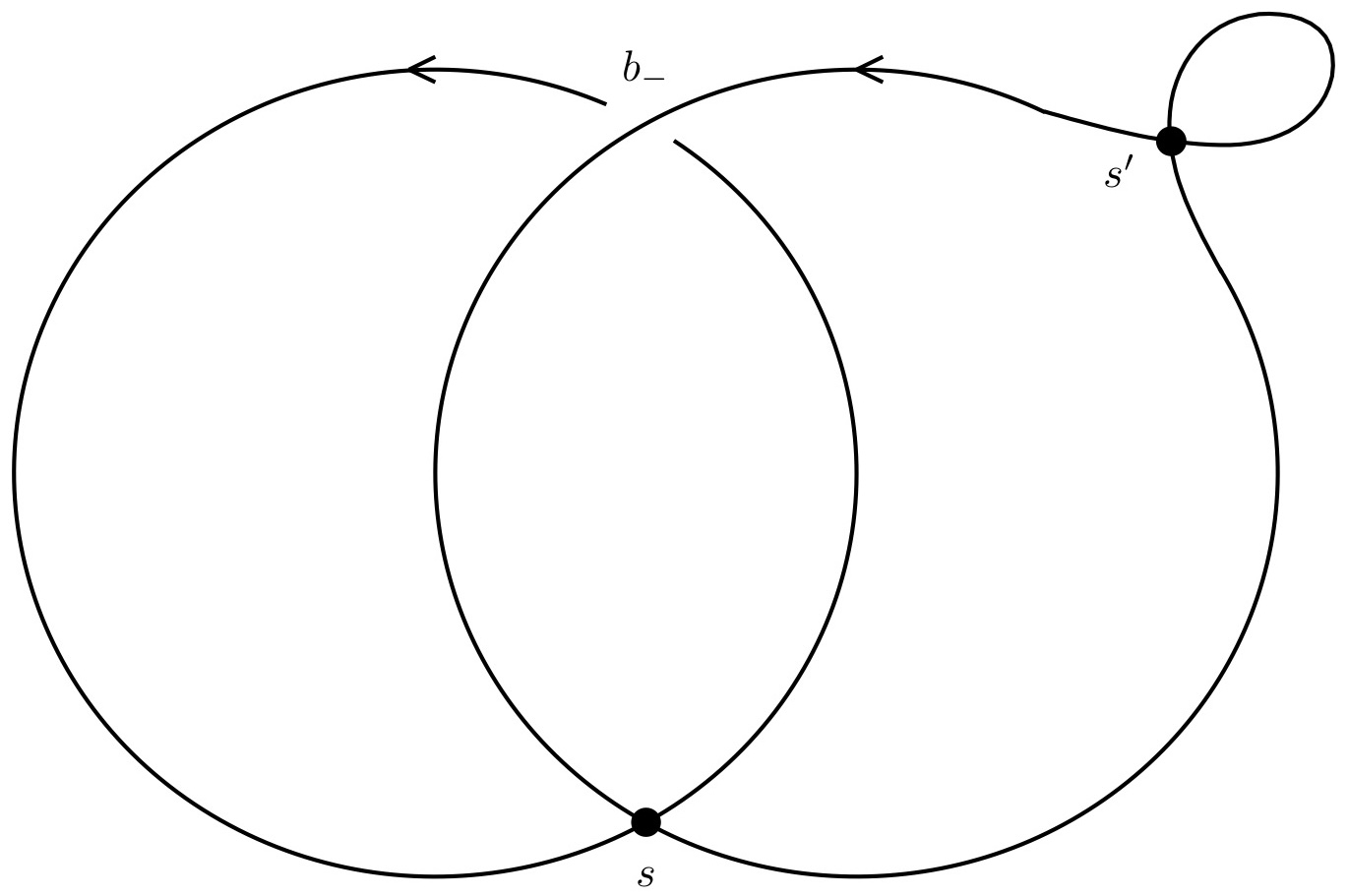}}
\newcommand{\ExampleO}{\includegraphics[width=0.32\linewidth]{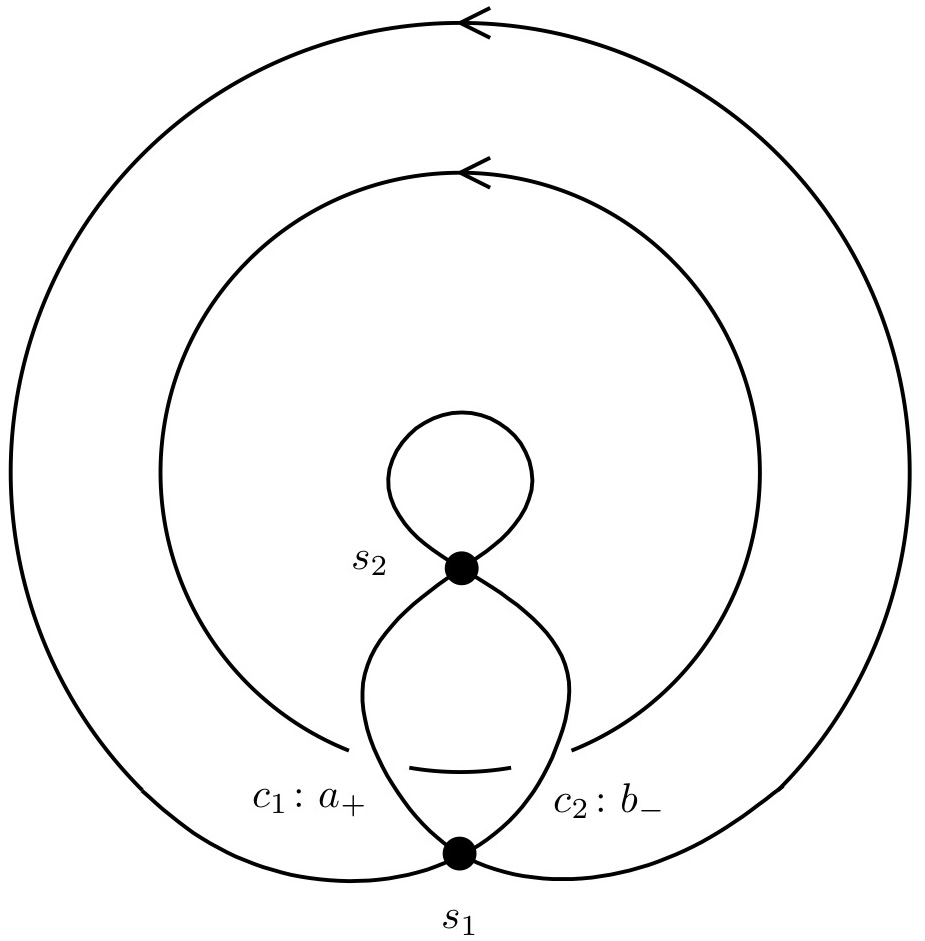}}
\newcommand{\ExampleU}{\includegraphics[width=0.32\linewidth]{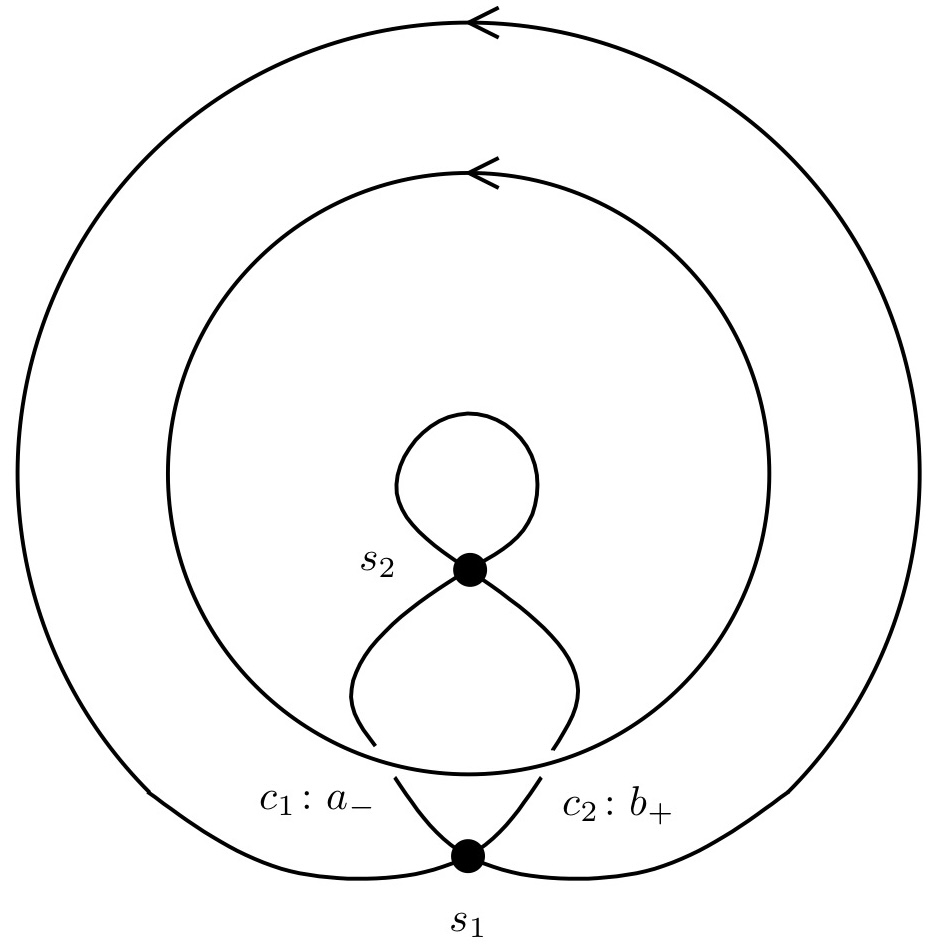}}

\title{Minimal Generating Sets of Singular Reidemeister Moves and Their Classification}
\author{Noboru Ito}
\address{Department of Mathematics, Faculty of Engineering, Shinshu University, Wakasato 4-17-1, Nagano, Nagano 380-8553, Japan}
\email{nito@shinshu-u.ac.jp}
\author{Yuichiro Iwamoto}
\address{Department of Science and Technology, Graduate School of Medicine, Science and Technology,
Shinshu University, Asahi 3-1-1, Matsumoto, Nagano 390-8626, Japan}
\email{IwamotoY.math@gmail.com}
\date{June 11, 2026}

\begin{document}
\maketitle
\begin{abstract}
Singular knot theory extends classical knot theory by allowing transverse
double points without over/under information, together with singular
Reidemeister moves of types~$\IV$ and~$\V$.
A central open problem in this theory is to determine the minimal
generating sets of oriented singular Reidemeister moves.
In this paper, we completely solve this problem.
In addition, we establish independence results for singular Reidemeister moves
by introducing an invariant that provides obstructions
and lower bounds for generating sets,
including the independence of type~$\III$
from types~$\I$, $\II$, $\IV$, and $\V$.
More precisely, starting from a minimal generating set of ordinary
Reidemeister moves of types~$\I$--$\III$, we prove that the singular
moves admit exactly $96$ distinct inclusion-minimal generating sets,
and that these exhaust all possibilities.
Our proof introduces a new invariant for singular links,
constructed via a projection to self-singular links, which detects the
distinction between the two families of type~$\IV$ moves and provides
an obstruction for generating type~$\V$ moves from types~$\I$--$\IV$.
We also determine the unoriented case, where the classification
collapses to exactly $8$ minimal generating sets.
\end{abstract}
\section{Introduction}\label{Intro}

A \emph{singular knot}, \emph{singular link}, or \emph{singular tangle} is defined as an immersion of a circle (knot), a collection of circles (link), or properly embedded arcs (tangle) in $\mathbb{R}^3$ that permits \emph{singular crossings} that are transverse double points without over/under information.  
Note that a knot, link, or tangle diagram with no singular crossings is just an ordinary knot, link, or tangle diagram.
(i.e.,  embedding without singular crossings).    

This framework has found recent relevance in modeling molecular structures, notably in biology \cite{ADEM2020}.  For example, in the representation of proteins, hydrogen bonds, which are key intra-chain contacts, fit naturally as singular crossings within projected diagrams. Therefore,  singular knot theory has been extended to account for these intra-chain interactions, significantly broadening its applicability to otherwise unknotted proteins.   

In many biological models of singular links, moves of type~$\V$
naturally arise and are therefore included as part of the standard
equivalence relation.  The presence of type~$\V$ significantly
complicates the classification problem, since this move changes the labels assigned to inter-component crossings
and thus invalidates several naive diagrammatic invariants. Consequently, determining minimal generating sets in the
presence of type~$\V$ becomes substantially more difficult than in the
case involving only types~$\I$--$\IV$.

Equivalence between two singular knots, links, or tangles is given by \emph{singular Reidemeister moves}: types $\I$, $\II$, $\III$ (\emph{ordinary} Reidemeister moves from non-singular theory), and $\IV$, $\V$ that specifically account for singular crossings.  

Two singular diagrams represent the same object if one can be transformed into the other by a finite sequence of these moves (and planar isotopies).  
Here, singular Reidemeister moves are local replacements applied to neighborhoods of diagrams, each supported inside an oriented embedded disk in the plane, called a \emph{changing disk}:     
\begin{figure}[h]
        \centering
        \includegraphics[width=0.75\linewidth]{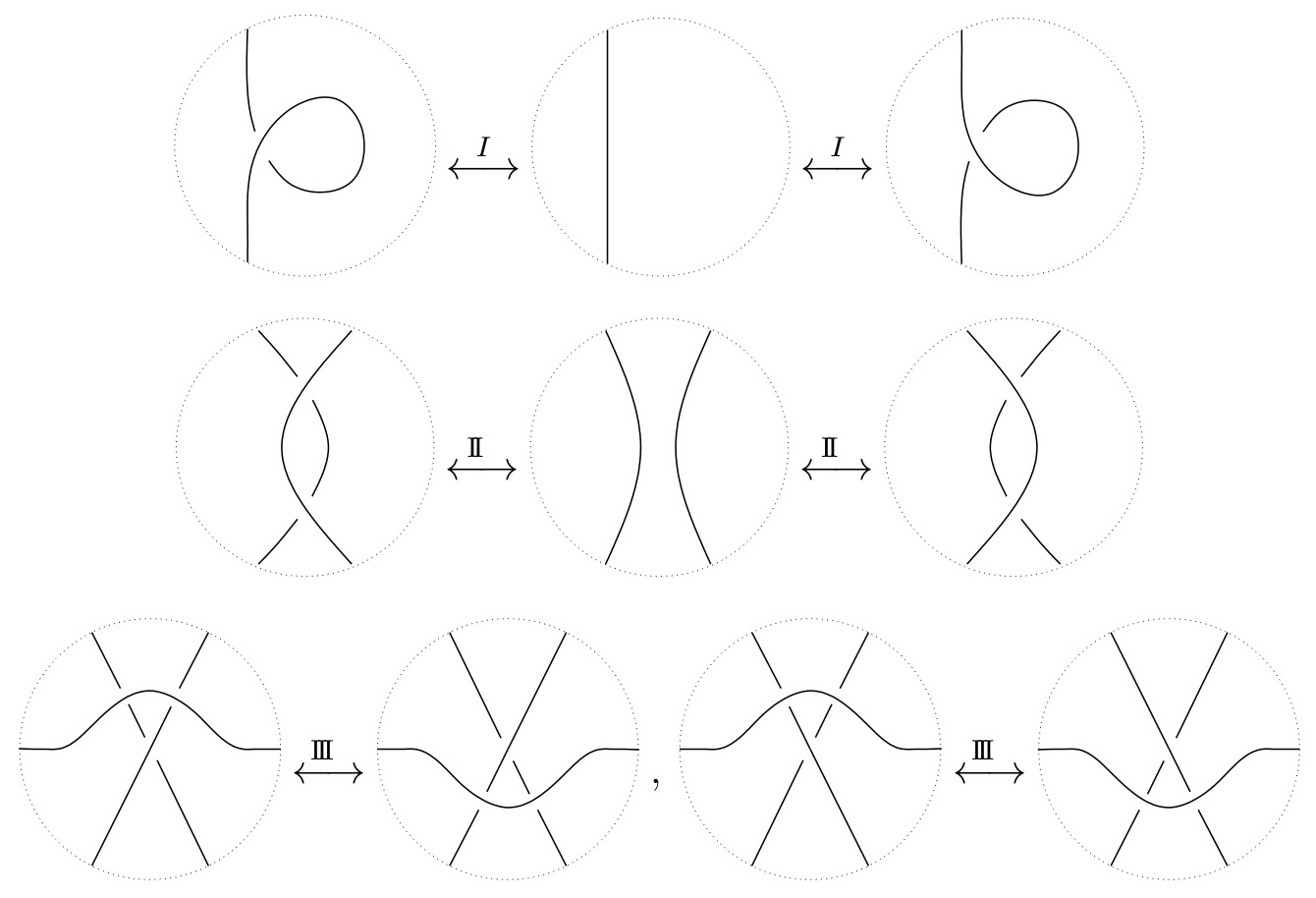}
\end{figure}
\begin{figure}[h]
        \centering
        \includegraphics[width=0.75\linewidth]{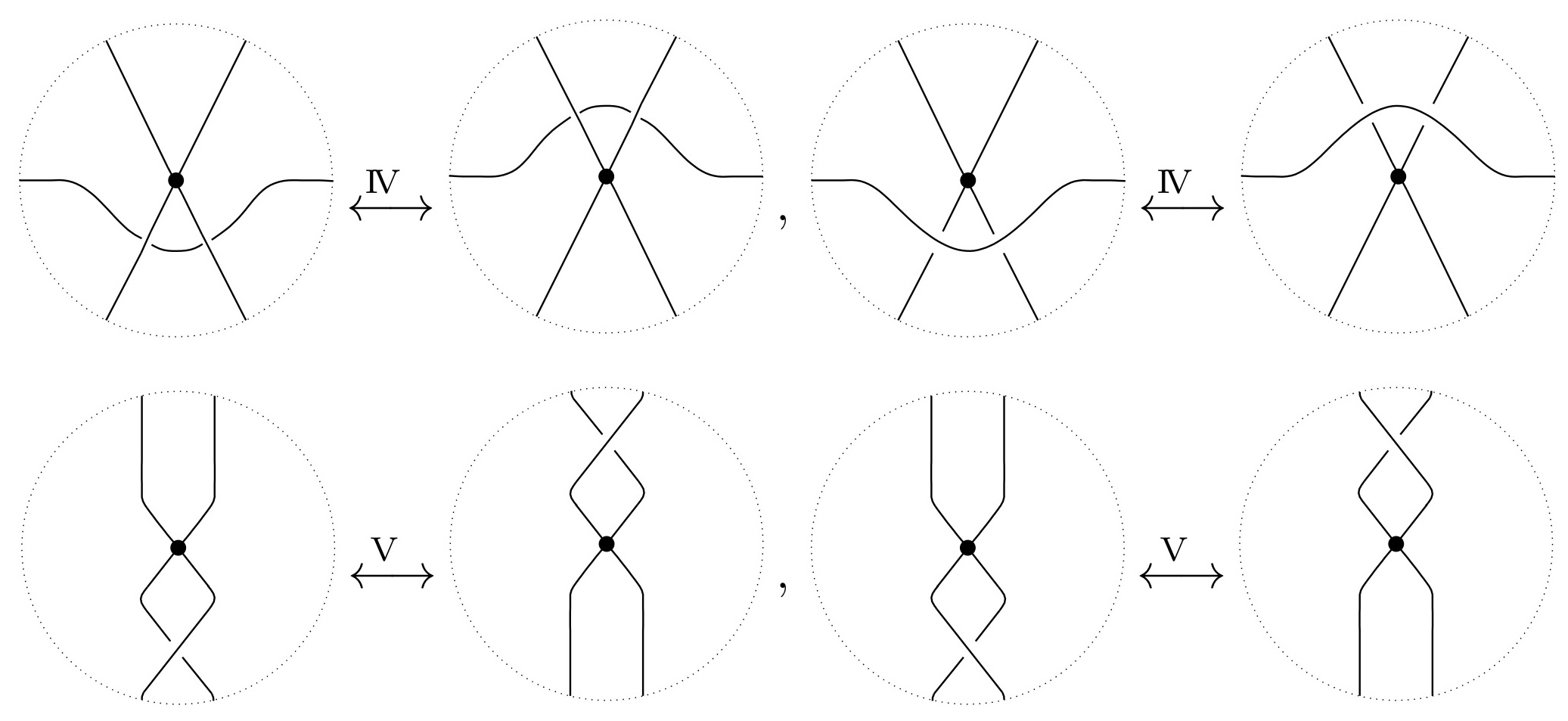}
\end{figure}






Building on this foundation, Bataineh-Elhamdadi-Hajij-Youmans \cite{BEHY2018} introduced a specific generating set for oriented singular Reidemeister moves in their work on singquandles, and posed several open problems.  One of the central open problems posed in \cite{BEHY2018} is the following. 
\begin{prob}[{\cite[The first problem in open questions (Section~6)]{BEHY2018}}]\label{prob:smgs}
Find other generating sets of oriented singular Reidemeister moves and prove their minimality.  
\end{prob}
Here minimality is understood in the inclusion sense: removing any one move loses the generating property.  

In this paper, we completely solve this problem.  
In particular, our result not only produces new generating sets,
but gives a complete classification of all inclusion-minimal
generating sets, thereby resolving the “other” aspect of
Problem~\ref{prob:smgs} in full.

The key difficulty is the presence of type~$\V$ moves.
Unlike types~$\I$--$\IV$, a move of type~$\V$ changes the labels of inter-component crossings,
so several naive diagrammatic invariants are no longer available.
While generating sets of oriented singular Reidemeister moves have been considered in the literature, even the basic structure of minimal generating sets has remained unknown.
A key ingredient of our proof is a new diagrammatic invariant
for singular links, based on a natural projection from the
category of singular link diagrams to that of self-singular
link diagrams obtained by resolving inter-component singular
crossings.
This invariant is preserved under Reidemeister moves of
types~$\I$--$\III$ and~$\V$, while detecting the distinction
between the two types of type~$\IV$ moves.  
This observation leads to an invariant-based obstruction:
the labels of inter-component crossings detect that moves of type~$\V$
cannot be generated by moves of types~$\I$--$\IV$,
a fact that will be used to establish the minimality results below. 
We have completely solved Problem~\ref{prob:smgs}. 
We list the minimal generating sets (Table~\ref{ListSMGS}) 
and prove their minimality for oriented singular Reidemeister moves
for knots, links, and tangles.
Our result provides the first complete determination of
inclusion-minimal generating sets of $\R \cup \SR$
in the presence of type~$\V$ moves.
In this sense, our result may be regarded as a singular analogue of
Polyak's classification of minimal generating sets for the ordinary
Reidemeister moves.

\begin{thm}\label{thm:mainR}
Let $M$ be a minimal generating set of the ordinary Reidemeister moves $\R$.
Then the inclusion-minimal generating sets of $\R \cup \SR$
are exactly the sets of the form
\[
M \cup \{ m_{\IV}^{(O)},\, m_{\IV}^{(U)},\, m_{\V} \},
\]
where $m_{\IV}^{(O)} \in \RIVO$, $m_{\IV}^{(U)} \in \RIVU$, and $m_{\V} \in \RV$.
In particular, there are
\[
 4 \times 4 \times 6 = 96 
\]
such minimal generating sets (Table~\ref{ListSMGS}).
\end{thm}
This description makes explicit the structural decomposition of minimal generating sets.  
In particular, the classification reflects the decomposition of singular Reidemeister moves into two independent families of type~$\IV$ moves and one family of type~$\V$ moves.

To illustrate, choose one from each of the two unoriented types of singular Reidemeister moves of type~\IV, and one from the unoriented types of moves of type~\V, and assign arbitrary orientations to these three moves.  The resulting set is an example of the minimal generating sets as described in Theorem~\ref{thm:mainR}.  By repeating this process for all possible choices, we obtain all minimal generating sets described in Theorem~\ref{thm:mainR}.  
In the rest of this paper, we refer to such minimal generating sets of singular Reidemeister moves as \emph{singular minimal generating sets}.   

The proof proceeds in two steps.
First, we determine which subsets of type~$\IV$ moves can generate all
type~$\IV$ moves together with the ordinary Reidemeister moves.
Second, we analyze the interaction with type~$\V$ moves using an invariant
constructed in Section~\ref{sec:invariant}.
Combining these results yields the complete classification stated in
Theorem~\ref{thm:mainR}.
\begin{table}[t]
\caption{The $96$ inclusion-minimal generating sets of singular Reidemeister moves.}
\label{ListSMGS}
\begin{center}
\begin{tabular}{|c|c|c|c|c|c|} 
    \hline
   $\{\fiveA, \fourA, \fourE \}$ & $\{\fiveB, \fourA, \fourE \}$ & $\{\fiveC, \fourA, \fourE \}$  &  $\{\fiveD, \fourA, \fourE \}$  & $\{\fiveE, \fourA, \fourE \}$ & $\{\fiveF, \fourA, \fourE \}$       \\
   $\{\fiveA, \fourA, \fourF \}$ & $\{\fiveB, \fourA, \fourF \}$ & $\{\fiveC, \fourA, \fourF \}$  &  $\{\fiveD, \fourA, \fourF \}$  & $\{\fiveE, \fourA, \fourF \}$ & $\{\fiveF, \fourA, \fourF \}$       \\
   $\{\fiveA, \fourA, \fourG \}$ & $\{\fiveB, \fourA, \fourG \}$ & $\{\fiveC, \fourA, \fourG \}$  &  $\{\fiveD, \fourA, \fourG \}$  & $\{\fiveE, \fourA, \fourG \}$ & $\{\fiveF, \fourA, \fourG \}$       \\
   $\{\fiveA, \fourA, \fourH \}$ & $\{\fiveB, \fourA, \fourH \}$ & $\{\fiveC, \fourA, \fourH \}$  &  $\{\fiveD, \fourA, \fourH \}$  & $\{\fiveE, \fourA, \fourH \}$ & $\{\fiveF, \fourA, \fourH \}$       \\
   
   $\{\fiveA, \fourB, \fourE \}$ & $\{\fiveB, \fourB, \fourE \}$ & $\{\fiveC, \fourB, \fourE \}$ & $\{\fiveD, \fourB, \fourE \}$ & $\{\fiveE, \fourB, \fourE \}$ & $\{\fiveF, \fourB, \fourE \}$       \\
   $\{\fiveA, \fourB, \fourF \}$ & $\{\fiveB, \fourB, \fourF \}$ & $\{\fiveC, \fourB, \fourF \}$ & $\{\fiveD, \fourB, \fourF \}$ & $\{\fiveE, \fourB, \fourF \}$ & $\{\fiveF, \fourB, \fourF \}$       \\
   $\{\fiveA, \fourB, \fourG \}$ & $\{\fiveB, \fourB, \fourG \}$ & $\{\fiveC, \fourB, \fourG \}$ & $\{\fiveD, \fourB, \fourG \}$ & $\{\fiveE, \fourB, \fourG \}$ & $\{\fiveF, \fourB, \fourG \}$       \\
   $\{\fiveA, \fourB, \fourH \}$ & $\{\fiveB, \fourB, \fourH \}$ & $\{\fiveC, \fourB, \fourH \}$ & $\{\fiveD, \fourB, \fourH \}$ & $\{\fiveE, \fourB, \fourH \}$ & $\{\fiveF, \fourB, \fourH \}$       \\
   
   $\{\fiveA, \fourC, \fourE \}$ & $\{\fiveB, \fourC, \fourE \}$ & $\{\fiveC, \fourC, \fourE \}$ & $\{\fiveD, \fourC, \fourE \}$ & $\{\fiveE, \fourC, \fourE \}$ & $\{\fiveF, \fourC, \fourE \}$       \\
   $\{\fiveA, \fourC, \fourF \}$ & $\{\fiveB, \fourC, \fourF \}$ & $\{\fiveC, \fourC, \fourF \}$ & $\{\fiveD, \fourC, \fourF \}$ & $\{\fiveE, \fourC, \fourF \}$ & $\{\fiveF, \fourC, \fourF \}$       \\
   $\{\fiveA, \fourC, \fourG \}$ & $\{\fiveB, \fourC, \fourG \}$ & $\{\fiveC, \fourC, \fourG \}$ & $\{\fiveD, \fourC, \fourG \}$ & $\{\fiveE, \fourC, \fourG \}$ & $\{\fiveF, \fourC, \fourG \}$       \\
   $\{\fiveA, \fourC, \fourH \}$ & $\{\fiveB, \fourC, \fourH \}$ & $\{\fiveC, \fourC, \fourH \}$ & $\{\fiveD, \fourC, \fourH \}$ & $\{\fiveE, \fourC, \fourH \}$ & $\{\fiveF, \fourC, \fourH \}$       \\
   
   $\{\fiveA, \fourD, \fourE \}$ & $\{\fiveB, \fourD, \fourE \}$ & $\{\fiveC, \fourD, \fourE \}$  & $\{\fiveD, \fourD, \fourE \}$ & $\{\fiveE, \fourD, \fourE \}$ & $\{\fiveF, \fourD, \fourE \}$       \\
   $\{\fiveA, \fourD, \fourF \}$ & $\{\fiveB, \fourD, \fourF \}$ & $\{\fiveC, \fourD, \fourF \}$  & $\{\fiveD, \fourD, \fourF \}$ & $\{\fiveE, \fourD, \fourF \}$ & $\{\fiveF, \fourD, \fourF \}$       \\
   $\{\fiveA, \fourD, \fourG \}$ & $\{\fiveB, \fourD, \fourG \}$ & $\{\fiveC, \fourD, \fourG \}$  & $\{\fiveD, \fourD, \fourG \}$ & $\{\fiveE, \fourD, \fourG \}$ & $\{\fiveF, \fourD, \fourG \}$       \\
   $\{\fiveA, \fourD, \fourH \}$ & $\{\fiveB, \fourD, \fourH \}$ & $\{\fiveC, \fourD, \fourH \}$  & $\{\fiveD, \fourD, \fourH \}$ & $\{\fiveE, \fourD, \fourH \}$ & $\{\fiveF, \fourD, \fourH \}$       \\
\hline
\end{tabular}
\end{center}
\end{table}
\begin{rem}
Related configurations of singular Reidemeister moves
have appeared in the context of Legendrian knots
(see \cite{YamaguchiIto2022}), although minimality and complete classification  
were not considered there.
\end{rem}

We further extend our analysis to the unoriented setting,
where the classification collapses to $8$ minimal generating sets.
This appears to be the first complete determination in the
unoriented singular setting.
We denote by $\RI^{\!\mathrm{un}}, \RII^{\!\mathrm{un}}, \RIII^{\!\mathrm{un}}, 
\RIVO^{\mathrm{un}}, \RIVU^{\mathrm{un}}, \RV^{\!\mathrm{un}}$
the sets of unoriented Reidemeister moves corresponding to 
$\RI, \RII, \RIII, \RIVO, \RIVU, \RV$, respectively.
\begin{thm}\label{thm:unoriented}
The minimal generating sets of unoriented singular Reidemeister moves
consist of exactly one move from each of the sets
$\RI^{\!\mathrm{un}}, \RII^{\!\mathrm{un}}, \RIII^{\!\mathrm{un}},
\RIVO^{\mathrm{un}}, \RIVU^{\mathrm{un}}, \RV^{\!\mathrm{un}}$.
In particular, there are $2^3 = 8$ such minimal generating sets.
\end{thm}

\clearpage

\section{Definitions and notations}\label{def}

\subsection{Definitions}

\begin{defn}\label{def:symbolRM}
    Oriented Reidemeister moves are defined as follows.\\
    \vspace{-0.5cm}
    \begin{figure}[h]
        \centering
        \includegraphics[width=1\linewidth]{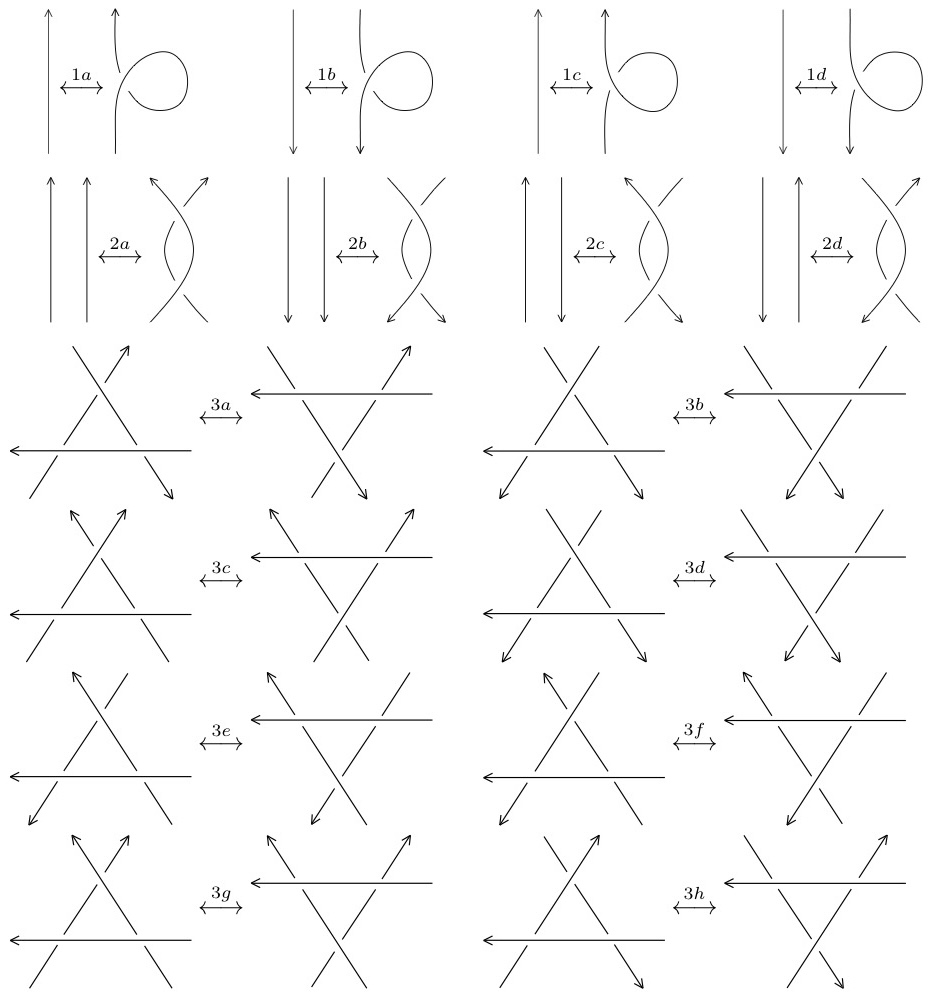}
        \caption{Oriented (ordinary)  Reidemeister moves}       \label{fig:symbolRM}
    \end{figure}
    \vspace{-0.3cm}
\end{defn}
        Type \III~moves are classified into braid type and non-braid type: $\threeA$ and $\threeH$ are non-braid type, while $\threeB, \threeC, \threeD, \threeE, \threeF, \threeG$ are braid type.
    Similarly, among the \tII~moves, $\twoC$ and $\twoD$ are non-braid type, and $\twoA, \twoB$ are braid type.
\begin{defn}
    We denote by $\RI, \RII, \RIII$ the set of all moves of \tI, \tII, \tIII.
    And let $\R \coloneq \RI \cup \RII \cup \RIII$. 
\end{defn}    

\clearpage

\begin{defn}\label{def:symbolSRM}
    Oriented singular Reidemeister moves are defined as follows.\\
\begin{figure}[h]
        \centering
        \includegraphics[width=1\linewidth]{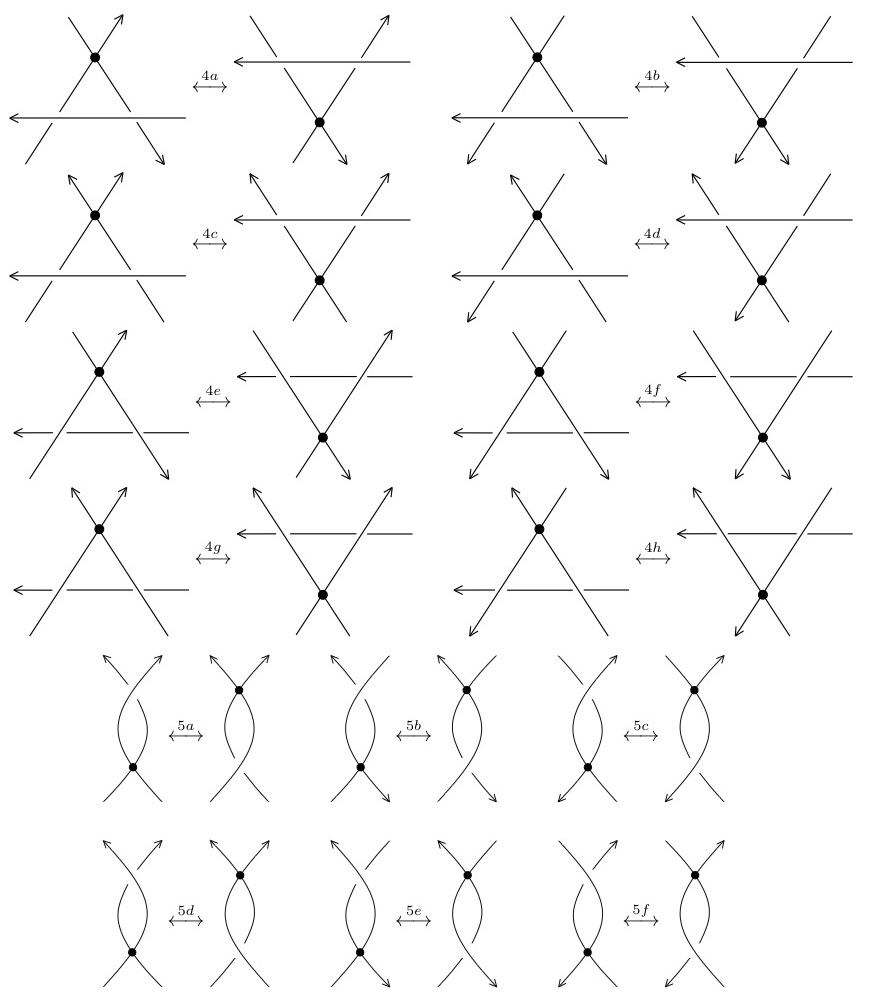}
\caption{Oriented  singular Reidemeister moves}
\label{fig:symbolSRM}
\end{figure}
\end{defn}
\vspace{-0.3cm}

For convenience, we record the correspondence between our labeling
and that of \cite{BEHY2018} (Table~\ref{table:CorrespondingFour}).
\begin{table}[h]
\centering
\caption{Correspondence between our notation and that of \cite{BEHY2018}.}\label{table:CorrespondingFour}
\begin{tabular}{c}
\qquad\qquad\ \ \cite{BEHY2018}: $4a,\ 4d,\ 4c,\ 4b \mid 4e,\ 4h,\ 4g,\ 4f$ \\ \hline
Our notation: $4a,\ 4b,\ 4c,\ 4d \mid 4e,\ 4f,\ 4g,\ 4h$
\end{tabular}
\end{table}
Type \IV~moves are classified into braid type and non-braid type: $\fourA$ and $\fourE$ are non-braid type, while $\fourB, \fourC, \fourD, \fourF, \fourG, \fourH$ are braid type.
    Similarly, among the \tV~moves, $\fiveB$, $\fiveC$, $\fiveE$ and $\fiveF$ are non-braid type, and $\fiveA, \fiveD$ are braid type.
\begin{defn}
    We denote by $\RIV$, $\RV$ the set of all moves of \tIV, \tV.
    And let $\SR \coloneq \RIV \cup \RV$. 
\end{defn}    

\begin{defn}
    Let $\RIVO \coloneq \{ \fourA, \fourB, \fourC, \fourD \}$ and $\RIVU \coloneq \{ \fourE, \fourF, \fourG, \fourH \}$. 
\end{defn}    

\begin{defn}\label{generated}
Let \(m\) be a local move supported in a disk $B$. We say that $m$ is generated by a set  $S$ of local moves if the local change prescribed by $m$ can be achieved by a finite sequence of moves from $S$, all performed within the same disk $B$, 
without affecting 
the diagram outside $B$. 
In this case, we write
\[
m \prec S.
\]
\end{defn}


\begin{defn}
A set $S \subset \R$ is called a \emph{generating set of $\R$} if for any $m \in \R$, $m \prec S$.
We denote by $\GSordi$ the family of all generating sets of $\R$.
\end{defn}

\begin{defn}\label{def:minimal}
    The set $S \in \GSordi$ is a \emph{minimal generating set} 
    if for any $m \in S$, $ S \setminus \{m\} \notin \GSordi $.  
    We denote by $\MGSordi$ the set of all minimal generating sets of $\R$.
\end{defn}


\begin{defn}
A set $S \subset \R \cup \SR$ is called a \emph{generating set of $\R \cup \SR$} if for any $m \in \R \cup \SR$, $m \prec S$.
We denote by $\GSsingu$ the family of all generating sets of $\R \cup \SR$.
\end{defn}

\begin{defn}\label{def:Sminimal}
    The set $S \in \GSsingu$ is a \emph{minimal generating set} 
    if for any $m \in S$, $ S \setminus \{m\} \notin \GSsingu $.  
    We denote by $\MGSsingu$ the set of all minimal generating sets of $\R \cup \SR$.
\end{defn}

\subsection{Independence of type~\III}
We establish several non-generation results for Reidemeister moves,
culminating in the independence of type~$\IIIno$.

\begin{prop}\label{prop:I-II-non-gene}
Ordinary Reidemeister moves of types~$\Ino$ and $\IIno$
cannot be generated by singular Reidemeister moves of types~$\IVno$ and $\Vno$.
\end{prop}

\begin{proof}
Moves of types~$\Ino$ and $\IIno$ change the number of real crossings,
whereas moves of types~$\IVno$ and $\Vno$ preserve it.
Hence no move of $\tI$ or $\tII$ can be generated by moves of types~$\IVno$ and $\Vno$.
\end{proof}

For a singular link diagram $D$, let
\[
\prs(D)=D_s
\]
denote the diagram obtained from $D$ by smoothing all singular crossings.

\begin{prop}\label{prop:prs}
The map $\prs : D \mapsto D_s$ is well defined up to Reidemeister moves
of types $\Ino$ and $\IIno$.
\end{prop}

\begin{proof}
It suffices to check the behavior under moves of types $\IVno$ and $\Vno$.

For a move of $\tIV$, after smoothing the singular crossing,
the resulting diagrams differ by a Reidemeister move of $\tII$
or by a planar isotopy.

For a move of $\tV$, after smoothing, the resulting diagrams differ
by a Reidemeister move of $\tI$
or by a planar isotopy.

Hence $\prs(D)$ is well defined up to Reidemeister moves of types $\Ino$ and $\IIno$.
\end{proof}

We now prove the independence of $\tIII$ by reducing to the ordinary case.
We use \"{O}stlund's Gauss diagram formula \cite{Ostlund2001}, which is invariant under Reidemeister moves of types~$\Ino$ and $\IIno$, but changes under at least one move of type~$\IIIno$.

\begin{thm}\label{thm:RIII-non-gene}
A move of $\tIII$ cannot be generated by moves of types
$\Ino$, $\IIno$, $\IVno$, and $\Vno$.
\end{thm}

\begin{proof}
Assume that a move of type~$\IIIno$ can be generated by moves of types
$\Ino$, $\IIno$, $\IVno$, and $\Vno$.

Applying $\prs$, any sequence of moves of types $\Ino$, $\IIno$, $\IVno$, and $\Vno$
induces a sequence of Reidemeister moves of types $\Ino$ and $\IIno$
on the ordinary diagram $\prs(D)$ by Proposition~\ref{prop:prs}.

On the other hand, \"{O}stlund's Gauss diagram formula is invariant under moves of types $\Ino$ and $\IIno$, and changes under a move of type~$\IIIno$.
It therefore detects a type~$\IIIno$ move that cannot be realized by moves of types $\Ino$ and $\IIno$ alone.

This contradicts the assumption.
Hence a move of type~$\IIIno$ cannot be generated by moves of types
$\Ino$, $\IIno$, $\IVno$, and $\Vno$.
\end{proof}

\begin{cor}
A necessary condition for a set $S$ to be a minimal generating set of $\R \cup \SR$
is that $S$ contains a minimal generating set of $\R$.
\end{cor}

\begin{proof}
Let $S$ be a minimal generating set of $\R \cup \SR$.
Then $S$ generates all moves in $\R$.
By Proposition~\ref{prop:I-II-non-gene} and Theorem~\ref{thm:RIII-non-gene},
none of the moves of types $\Ino$, $\IIno$, and $\IIIno$ can be generated solely by singular moves.
Hence $S$ must contain a subset generating $\R$.
Taking such a subset minimal, we obtain a minimal generating set of $\R$ contained in $S$.
\end{proof}

\section{Main results}\label{sec:MainR}

Let $M \in \MGSordi$ be a minimal generating set of the ordinary Reidemeister moves $\R$.


\begin{thm}\label{main-thm}
A set \(S\) is a minimal generating set of \(\R \cup \SR\)
if and only if there exist
\(M \in \MGSordi\),
\(\fourO \in \R\IVO\),
\(\fourU \in \R\IVU\),
and \(\fiveST \in \R\V\)
such that
\[
S = M \cup \{\fourO,\fourU,\fiveST\}.
\]
\end{thm}

Theorem~\ref{main-thm} follows from the following four propositions.

\begin{prop}\label{4-gene}
The following two statements hold.
\begin{itemize}
    \item Let $\ast \in \{a,b,c,d\}$ (i.e.\ $\fourAS \in \R \IVO$). Then the set
    \[
    \{\fourAS\}\cup M
    \]
    generates $\R \IVO$.
    \item Let $\ast \in \{e,f,g,h\}$ (i.e.\ $\fourAS \in \R \IVU$). Then the set
    \[
    \{\fourAS\}\cup M
    \]
    generates $\R \IVU$.
\end{itemize}
\end{prop}

\begin{prop}\label{4-non-gene}
The following two statements hold.
\begin{itemize}
    \item The set $\R \cup \R \IVO \cup \R \V$ never generates $\R \IVU$.
    \item The set $\R \cup \R \IVU \cup \R \V$ never generates $\R \IVO$.
\end{itemize}
\end{prop}

\begin{prop}\label{5-gene}
For any $\fourO \in \R \IVO$, $\fourU \in \R \IVU$, and $\fiveST \in \R \V$, the set
\[
\{\fourO,\fourU,\fiveST\}\cup M
\]
generates $\R \V$.
\end{prop}

\begin{prop}\label{5-non-gene}
The set $\R \cup \R \IV$ never generates $\R \V$.
\end{prop}

\begin{cor}
For every set $X$ appearing in Table~\ref{ListSMGS}, the union
\[
M \cup X
\]
is a minimal generating set of singular Reidemeister moves $\R \cup \SR$,
and Table~\ref{ListSMGS} gives the complete list of all such sets $X$.
\end{cor}
We also establish the independence of type~$\III$
from types~$\I$, $\II$, $\IV$, and $\V$
(Theorem~\ref{thm:RIII-non-gene}),
which plays a key role in the minimality arguments.

The proofs are organized as follows.
Proposition~\ref{4-gene} is established in Subsection~\ref{subsec:gene-4},
Proposition~\ref{4-non-gene} follows from Propositions~\ref{prop:inv-no-IV}
and~\ref{prop:IVdetect},
Proposition~\ref{5-gene} is proved in Subsection~\ref{subsec:gene-5},
and Proposition~\ref{5-non-gene} is proved in Subsection~\ref{sec:Detect-5}.

\section{Proof}\label{Proof}

\subsection{Generating \tIV}\label{subsec:gene-4}
    We use the notation \(m \prec S\) introduced in Definition~\ref{generated}

\begin{lem}\label{Generating4}
    $\tIV$ moves can be generated as follows.
  \begin{align}
    \fourA &\prec \{ \fourB, \twoC, \twoD \}, \{ \fourC, \twoC, \twoD \},\\
    \fourB &\prec \{ \fourA, \twoC, \twoD \}, \{ \fourD, \twoA, \twoB \},\\
    \fourC &\prec \{ \fourA, \twoC, \twoD \}, \{ \fourD, \twoA, \twoB \},\\
    \fourD &\prec \{ \fourB, \twoA, \twoB \}, \{ \fourC, \twoA, \twoB \},\\
    \fourE &\prec \{ \fourF, \twoC, \twoD \}, \{ \fourG, \twoC, \twoD \},\\
    \fourF &\prec \{ \fourE, \twoC, \twoD \}, \{ \fourH, \twoA, \twoB \},\\
    \fourG &\prec \{ \fourE, \twoC, \twoD \}, \{ \fourH, \twoA, \twoB \},\\
    \fourH &\prec \{ \fourF, \twoA, \twoB \}, \{ \fourG, \twoC, \twoD \}.
  \end{align}
\end{lem}

\begin{proof}
    For the first case, $\fourA$ can be generated by the following sequence.\\ \\
    \vspace{-4mm}
    \begin{figure}[h]
        \centering
        \includegraphics[width=1\linewidth]{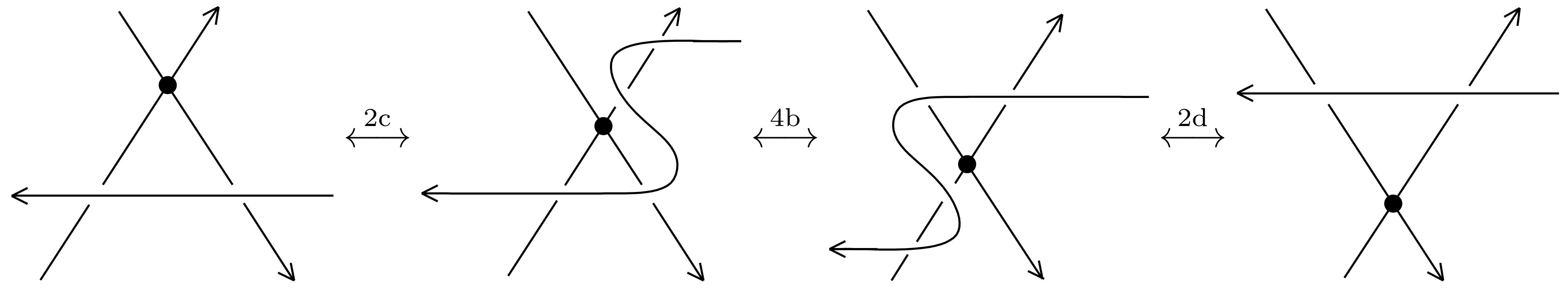}
    \end{figure}
    \vspace{-3mm}
    
    The same applies to the other $\tIV$ moves.
\end{proof}

\vspace{-3mm}
\begin{proof}[Proof of Proposition~\ref{4-gene}]
    For example, consider the case $\ast = a$.
    (The cases $\ast=b,c,d$ and $\ast=e,f, g, h$ can be treated in the same way.) 
    Since $M$ is a minimal generating set of ordinary Reidemeister moves, the moves $\twoA, \twoB, \twoC$, and $\twoD$ can all be generated using only elements of $M$. 
    Here, by the lemma~\ref{Generating4},  the moves $\fourB$ and $\fourC$ can first be generated.
    \begin{align*}
        \fourB, \fourC \prec \{ \fourA, \twoC, \twoD \}
    \end{align*}
    Using the resulting move $\fourB$ together with a type II move, we can then generate $\fourD$.
    \begin{align*}
        \fourD  \prec \{ \fourB, \twoA, \twoB \}
    \end{align*}
    Therefore, all elements of $\RIVO$ can be generated.
\end{proof}

\subsection{Generating \tV}\label{subsec:gene-5}
\begin{lem}\label{Generating5-1}
    $\tV$ moves can be generated as follows.
  \begin{align}
    \fiveA &\prec \{ \fiveE, \fourB, \fourG, \oneB \}, \{ \fiveF, \fourB, \fourG, \oneA \},\\
    \fiveB &\prec \{ \fiveD, \fourA, \fourE, \oneA \}, \\
    \fiveC &\prec \{ \fiveD, \fourA, \fourE, \oneA \}, \\
    \fiveD &\prec \{ \fiveB, \fourC, \fourF, \oneD \}, \{ \fiveC, \fourC, \fourF, \oneC \}.\\
    \fiveE &\prec \{ \fiveA, \fourA, \fourE, \oneC \},  \\
    \fiveF &\prec \{ \fiveA, \fourA, \fourE, \oneC \}. 
  \end{align}
\end{lem}

\begin{proof}
    For the first case, $\fiveA$ can be generated by the following sequence.\\ \\
    \vspace{-8mm}
    \begin{figure}[h]
        \centering
        \includegraphics[width=1\linewidth]{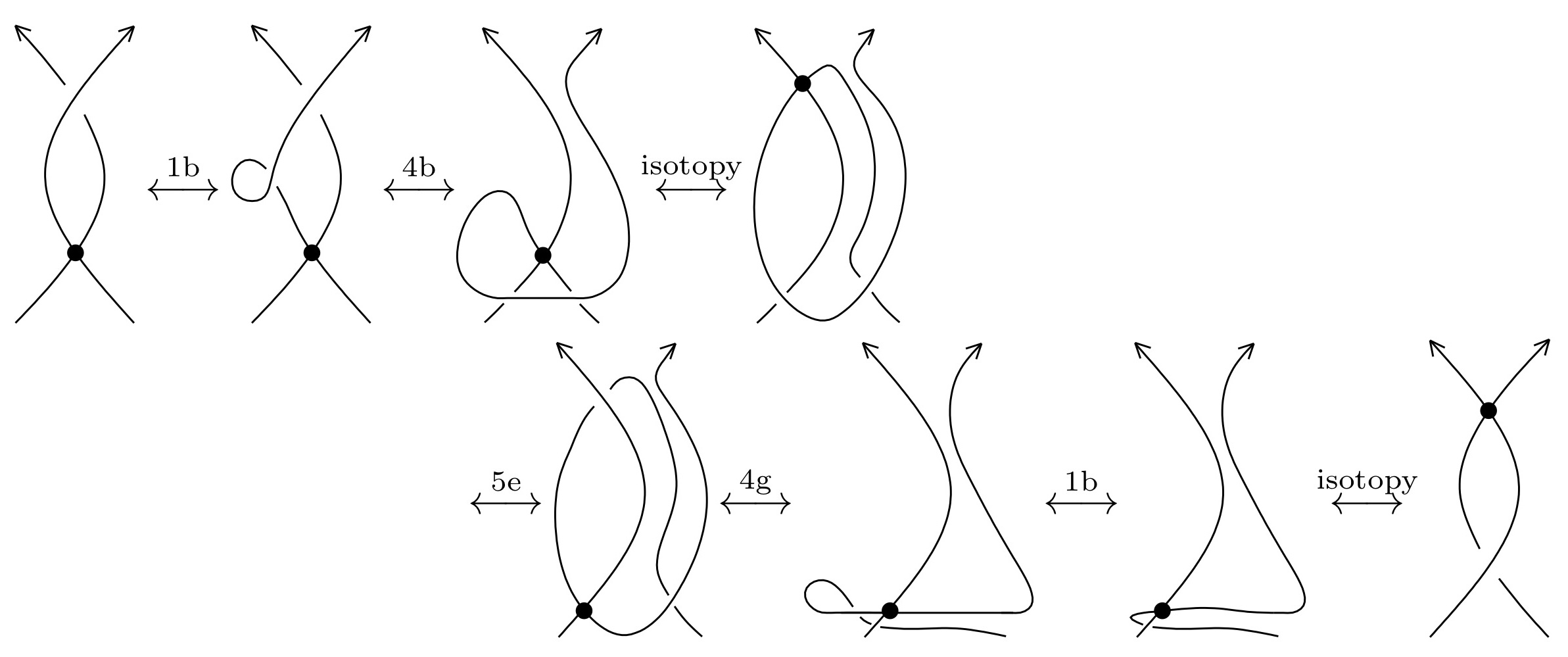}
    \end{figure}
    \vspace{-5mm}
    \\
    
    The same applies to the other $\tV$ moves.
    
\end{proof}

\begin{lem}\label{Generating5-2}
    $\tV$ moves can be generated as follows.
  \begin{align}
    \fiveA &\prec \{ \fiveD, \twoA, \twoB \},\\
    \fiveB &\prec \{ \fiveF, \twoC \},\\
    \fiveC &\prec \{ \fiveE, \twoD \},\\
    \fiveD &\prec \{ \fiveA, \twoA, \twoB \},\\
    \fiveE &\prec \{ \fiveC, \twoC \},\\
    \fiveF &\prec \{ \fiveB, \twoD \}.
  \end{align}
\end{lem}

\begin{proof}
    For the first case, $\fiveA$ can be generated by the following sequence.
    \\ \\
    \begin{figure}[h]
        \centering
        \includegraphics[width=0.8\linewidth]{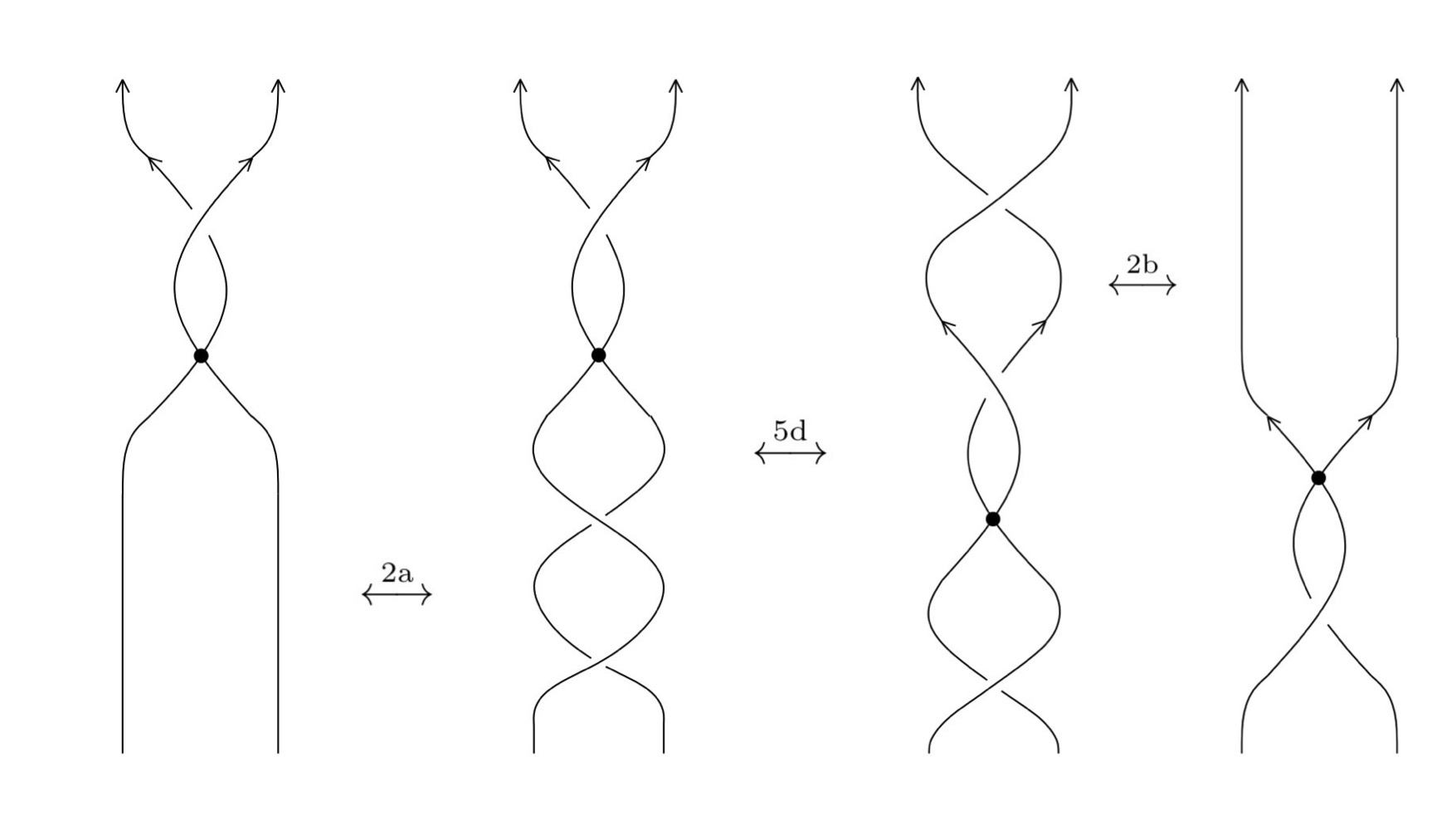}
    \end{figure}
    \vspace{-5mm}
    
    The same applies to the other $\tV$ moves.\\
\end{proof}

\begin{proof}[Proof of Proposition~\ref{5-gene}]
    For example, consider the case $\star = a$.
    (The cases $\ast=b, c, d, e, f$ can be treated in the same way.) 
    It follows from Proposition~\ref{4-gene} that, since $M$ is a minimal generating set of ordinary Reidemeister moves, any $\fourAS$, any $\fourU$, and the elements of $M$ together generate all the other moves of \tI, \II, \III, and \IV.
    Here, by the lemma~\ref{Generating5-1},  the moves $\fiveE$ and $\fiveF$ can first be generated.
    \begin{align*}
        \fiveE, \fiveF \prec \{ \fiveA, \fourA, \fourE, \oneC \}
    \end{align*}
    Moreover, by Lemma~\ref{Generating5-2}, the move $\fiveD$ can also be generated.
    \begin{align*}
        \fiveD  \prec \{ \fiveA, \twoA, \twoB \}
    \end{align*}
    It follows from Lemma~\ref{Generating5-1} that $\fiveB$ and $\fiveC$ can be generated from $\fiveD$.
    \begin{align*}
        \fiveB, \fiveC \prec \{ \fiveD, \fourA, \fourE, \oneA \}
    \end{align*}
    Therefore, all elements of $\RV$ can be generated.
\end{proof}

\subsection{An invariant detecting singular Reidemeister moves}\label{sec:invariant}
In this section, we introduce an invariant of based,
ordered two-component singular links, which plays a crucial role
in distinguishing singular Reidemeister moves of type~\IV. 
The main new ingredient is the invariant $f$, constructed from
Gauss phrases associated with the projection $\overline{\operatorname{pr}_+}$.
\subsubsection{Projection to self-singular links}\label{sec:projection}
A singular link is called a \emph{self-singular link} (resp.\ \emph{inter-component-singular link}) if every singular crossing is a self-singular crossing (resp.\ inter-component crossing).

Let $\mathcal{L}$ denote the set of singular links and $\mathcal{SL}$ the set of self-singular links.
Similarly, let $\mathcal{LD}$ denote the set of singular link diagrams and $\mathcal{SLD}$ the set of self-singular link diagrams.

\begin{prop}\label{prop:SL}
Define a map
\[
\overline{\operatorname{pr}_+} : \mathcal{LD} \longrightarrow \mathcal{SLD}
\]
by replacing every inter-component singular crossing in a diagram
with a positive crossing.

Then this map induces a well-defined map
\[
\operatorname{pr}_+ : \mathcal{L} \longrightarrow \mathcal{SL},
\]
that is, a projection from singular links to self-singular links.
\end{prop}

\begin{proof}
For type~\I,\II,\III, this $\overline{\operatorname{pr}_+}$ does not change $\I, \II, \III$, beacuse the singular crossing in the diagram are not replaced by positive crossing.   

For type $\IV$, this $\overline{\operatorname{pr}_+}$ does not change $\IV$ if a singular point is self-singular; $\overline{\operatorname{pr}_+}$ implies type $\III$ if a singular point is inter-component-singular. 

Finally, it suffices to check that the operation is compatible with the singular
Reidemeister move of type~V.

Let $D_0$ and $D_1$ be diagrams related by a type~V move.
By writing the diagrams explicitly before and after the move,
one verifies that $\overline{\operatorname{pr}_+}(D_0)$ and
$\overline{\operatorname{pr}_+}(D_1)$ represent the same singular link.
Hence the induced map $\operatorname{pr}_+$ is well defined.
\end{proof}

\begin{rem}
Let $\mathcal{IL}$ be the set of inter-component-singular links. 

Similarly, one may define another projection
\[
\operatorname{pr}' : \mathcal{L} \longrightarrow \mathcal{IL}
\]
by replacing all self-singular crossings simultaneously
with positive crossings.

This produces a different category of links.
Moreover, the two projections commute:
\[
\operatorname{pr}' \circ \operatorname{pr}_+
=
\operatorname{pr}_+ \circ \operatorname{pr}' .
\]
\end{rem}

\subsubsection{Setup and notation}\label{sec:labelAB}

We consider based, ordered two-component singular link diagrams $D$.
Throughout this section, all ordinary Reidemeister moves, types~\I--\III, 
are assumed to be freely available. 

Following Turaev \cite{Turaev2006}, we classify ordinary crossings between distinct components
into four types $a_+,a_-,b_+,b_-$ as follows: 
\[
\ap,\quad \am,\quad \bp,\quad \bm.  
\]
Here, in the diagrams of inter-component crossings,
the numerals $1$ and $2$ attached to the strands
indicate only the component to which each strand belongs
(the first or the second component, respectively).
These labels are independent of the over/under information. On the other hand, the sign $\pm$ indicates the local writhe.   
The crossing type $a_+, a_-, b_+, b_-$ is determined
by combining the component labels with the over/under structure.

\subsubsection{Encoding word}\label{subsec:encoding}

Let $D$ be an ordered, oriented two-component singular link diagram.
We refer to the components as the \emph{first} and the \emph{second} components.
Let $S$ denote the set of self-singular crossings on the first component.

For each $s\in S$, we define two Gauss phrases
$w_s^{\,o}$ and $w_s^{\,u}$.
A \emph{Gauss phrase} is a word in which every letter appears exactly twice,
together with separators $|$ indicating the decomposition into components.

\medskip
\noindent
\textbf{Step 1 (choice of base point).}
Fix $s\in S$.
Consider a small disk containing only the singular crossing $s$,
and regard $s$ as a positive crossing.

\begin{itemize}
\item For $w_s^{\,o}$, choose the base point on the first component
slightly before $s$ along the \emph{over-path}.
\item For $w_s^{\,u}$, choose the base point on the first component
slightly before $s$ along the \emph{under-path}.
\end{itemize}

On the second component, choose an arbitrary base point%
\footnote{The choice does not affect Propositions~\ref{prop:inv-no-IV} and \ref{prop:IVdetect}.}.
Each component is traversed from the chosen base point
following its orientation.

\medskip
\noindent
\textbf{Step 2 (construction of the Gauss phrases).}
Traverse each component starting from the chosen base point
following the orientation.
During this traversal record only the following events:

\begin{itemize}
\item[(S)] the two branches of the chosen self-singular crossing $s$,
\item[(I)] all inter-component ordinary crossings.
\end{itemize}

Each inter-component ordinary crossing is assigned a distinct label
$c_1,c_2,\dots,c_n$.
All other crossings are ignored.

Reading the recorded symbols along the traversal produces
the Gauss phrases
\[
w_s^{\,o},\; w_s^{\,u}
\]
in the alphabet
\[
\{\,s\,\}\cup\{\,c_1,c_2,\dots,c_n\,\},
\]
where the symbol $|$ separates the two components.

\medskip
\noindent
\textbf{Step 3 (separation numbers).}
Let $c$ be an inter-component ordinary crossing of $D$,
with type $a_+$, $a_-$, $b_+$, or $b_-$.

For each choice of a base point (on the over- or the under-path),
define
\[
\tau_o(s,c), \quad \tau_u(s,c)
\]
to be the number of occurrences of $s$
lying between the two occurrences of $c$
in the linear order of $w_s^{\,o}$ and $w_s^{\,u}$.
Let 
\begin{equation}\label{eq:sigmaBase}
\tau(s,c):=\tau_o(s,c)+\tau_u(s,c).   
\end{equation}

For a fixed type, define
\[
\tau_o(s,a_+)
:=
\sum_{\substack{c\ \text{inter-component}\\
\text{ordinary crossing of type } a_+}}
\tau_o(s,c),
\]
and similarly define
$\tau_u(s,a_+)$,
$\tau_o(s,a_-)$,
$\tau_u(s,a_-)$,
$\tau_o(s,b_+)$,
$\tau_u(s,b_+)$,
$\tau_o(s,b_-)$,
$\tau_u(s,b_-)$.

We then set
\[
\tau(s,a_+)
=
\tau_o(s,a_+)+\tau_u(s,a_+),
\]
and similarly for $a_-, b_+, b_-$.

Let $s'$ denote the unique inter-component singular crossing
appearing in the diagram, and let $c$ denote the unique
inter-component ordinary crossing.

\medskip
\noindent
\begin{ex}
For the diagram in Figure~\ref{fig:SHopfO}, the Gauss phrases are

\[
\begin{aligned}
w_s^{\,o} &= s'\,s'\,s\,c \mid c\,s,\\
w_s^{\,u} &= s'\,s\,c\,s' \mid c\,s.
\end{aligned}
\]

Hence
\[
\tau_o(s,c)=0, \qquad \tau_u(s,c)=1.
\]
\end{ex}

\begin{figure}
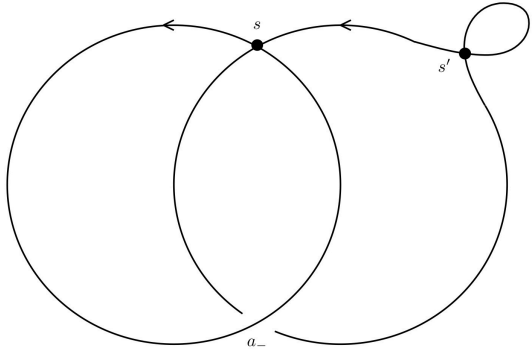

    \centering
    \SHopfO
    \caption{A singular link having an inter-component singular point $s$ and a self-singular point $s'$.  The label $a_-$ denotes the type of the crossing (cf.~Subsection~\ref{sec:labelAB}).  }
\label{fig:SHopfO}
\end{figure}

\medskip
\noindent
\begin{ex}
For the diagram in Figure~\ref{fig:SHopfU}, the Gauss phrases are

\[
\begin{aligned}
w_s^{\,o} &= s'\,s'\,c\,s \mid c\,s,\\
w_s^{\,u} &= s'\,c\,s\,s' \mid c\,s.
\end{aligned}
\]

Hence
\[
\tau_o(s,c)=0, \qquad \tau_u(s,c)=1.
\]
\end{ex}

\begin{figure}
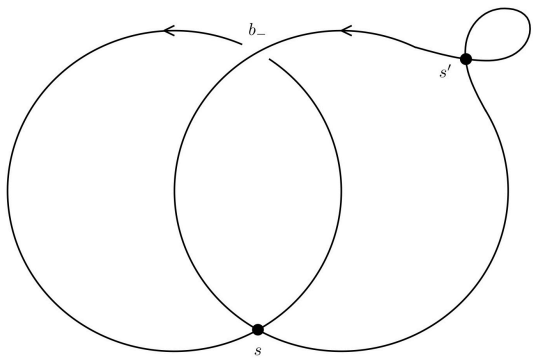

    \centering
    \SHopfU
    \caption{A singular link which is slightly different from the example as in Figure~\ref{fig:SHopfO}.  In particular, these two singular links are moved each other by a move of type $\V$.}
    \label{fig:SHopfU}
\end{figure}
\subsubsection{Definition of the invariant}\label{subsec:def-inv}

For each $s\in S$, define
\[
\tau(s,a_+) := \sum_{\substack{c \text{ inter-component}\\ \text{ordinary crossing of type } a_+}} \tau(s,c),
\quad
\tau(s,a_-) := \sum_{\substack{c \text{ inter-component}\\ \text{ordinary crossing of type } a_-}} \tau(s,c),
\]
and similarly define $\tau(s,b_+)$ and $\tau(s,b_-)$.

For an variant $t$, we associate to each $s\in S$ the ordered pair
\[
\Bigl(
t^{\,\tau(s,a_+)-\tau(s,b_-)},
\;
t^{\,\tau(s,a_-)-\tau(s,b_+)}
\Bigr) \in \mathbb{Z}[t] \times \mathbb{Z}[t].
\]   

Then, we define
\[
f(D)
:=
\left(
\sum_{s\in S} t^{\,\tau(s,a_+)-\tau(s,b_-)},
\;
\sum_{s\in S} t^{\,\tau(s,a_-)-\tau(s,b_+)}
\right) \in \mathbb{Z}[t] \times \mathbb{Z}[t],
\]
where the sum is taken over all self-singular crossings on  each component.   
\begin{ex}
Let $D$ be a diagram as in Figure~\ref{fig:SHopfO}.  
Using the invariant $f$ defined in
Subsection~\ref{subsec:def-inv}, we obtain
\[
f(D)=(1,t).
\]
\end{ex}
\begin{ex}
Let $D$ be a diagram as in Figure~\ref{fig:SHopfU}.  
Using the invariant $f$ defined in
Subsection~\ref{subsec:def-inv}, we obtain
\[
f(D)=(t^{-1},1).
\]
\end{ex}
\subsubsection{Invariance properties}\label{subsec:inv}

\begin{prop}\label{prop:inv-no-IV}
For any self-singular link diagram $D$, 
\begin{itemize}
\item The first entry of $f(D)$ is invariant under $\I$, $\II$, $\III$, and $\IVO$.
\item The second entry of $f(D)$ is invariant under $\I$, $\II$, $\III$, and $\IVU$. 
\end{itemize}
\end{prop}
\begin{proof}
A Reidemeister move of type~\I involves only self-ordinary crossings.
Such crossings are not inter-component ordinary crossings and hence do not contribute to the definition of $f(D)$.
Therefore $f(D)$ is unchanged.

A move of type~\II creates or removes a pair of inter-component ordinary crossings of types
$(a_+,b_-)$ or $(a_-,b_+)$.  Their contributions cancel in each summand, and thus $f(D)$ remains unchanged.

A move of type~\III only permutes ordinary crossings locally.
Since it does not change whether a given inter-component ordinary crossing separates the two occurrences of $s$ in $w_s$,
all values $\tau(s,c)$ are preserved, hence so is $f(D)$.

By definition, the first entry of $f(\overline{\operatorname{pr}_+}(D))$ is built from symbols $a_+$ and $b_-$,
while the second entry is built from $a_-$ and $b_+$.

In the Gauss-phrase  construction, a type~\IV move can change $\tau(s,c)$
only for those inter-component ordinary crossings whose types correspond to the strand
passing through the singular crossing.
For moves of type $\IVO$ (Definition~\ref{def:symbolSRM}, $4a$--$4d$),
only crossings of types $a_-$ and $b_+$ can affect the separation pattern with respect to $s$.
Hence the first entry (built from $a_+,b_-$) is preserved.
For moves of type $\IVU$ (Definition~\ref{def:symbolSRM}, $4e$--$4h$),
only crossings of types $a_+$ and $b_-$ can affect the separation pattern,
so the second entry (built from $a_-,b_+$) is preserved.
\end{proof}

\begin{ex}
    Consider applying $\overline{\operatorname{pr}_+}$ to the diagram of Figure~\ref{fig:SHopfO}.
    Then, under $\overline{\operatorname{pr}_+}$, the inter-component singular crossing $s$ is replaced by a positive crossing of type $b_+$.
    
    Hence 
    \[
        \tau(s, a_-) = 1, \qquad \tau(s, b_+) = 1.
    \]
    Therefore, we obtain
    \[
    f(\overline{\operatorname{pr}_+}(D))=(1,1).
    \]
\end{ex}

\begin{ex}
    Consider applying $\overline{\operatorname{pr}_+}$ to the diagram of Figure~\ref{fig:SHopfU}.
    Then, under $\overline{\operatorname{pr}_+}$, the inter-component singular crossing $s$ is replaced by a positive crossing of type $a_+$.
    
    Hence 
    \[
        \tau(s, b_-) = 1, \qquad \tau(s, a_+) = 1.
    \]
    Therefore, we obtain
    \[
    f(\overline{\operatorname{pr}_+}(D))=(1,1).
    \]
\end{ex}
\subsubsection{Detection of type~\IV moves}
\begin{prop}\label{prop:IVdetect}
Let $\IVO$ (resp.\ $\IVU$) denote a type~\IV move in which the over-path (resp.\ under-path)
passes through the singular crossing.
Let $D$ be an ordered, oriented two-component singular link diagram.  We have  
\begin{itemize}
\item The first entry of $f(\overline{\operatorname{pr}_+}(D))$ is invariant under $\I$, $\II$, $\III$, $\IVO$, and $\V$.
\item The second entry of $f(\overline{\operatorname{pr}_+}) (D)$ is invariant under $\I$, $\II$, $\III$, $\IVU$, and $\V$. 
\end{itemize}
\end{prop}

\begin{proof}
We will show the invariance under type~$\V$ moves.  
Note that $\overline{\operatorname{pr}_+}(D)$ contains only self-singular crossings.

Let $X$ and $Y$ denote subwords in the Gauss phrase.
Let $x$ (resp.\ $y$) be the number of occurrences of letters
corresponding to self-singular crossings in $X$ (resp.\ $Y$)
that contribute to the separation numbers.

Before applying a move of type~$\V$, the Gauss phrases take the form
\[
\text{(O)}\quad s\, Y\, s\, X \mid (\text{other component}), 
\qquad
\text{(U)}\quad s\, X\, s\, Y \mid (\text{other component}).
\]
The total contribution to $\tau$ is
\[
(2x + y + 1) + (x + 2y + 1) = 3x + 3y + 2.
\]

After applying a move of type~$\V$, the Gauss phrases become
\[
\text{(O)}\quad s\, X\, s\, Y \mid (\text{other component}), 
\qquad
\text{(U)}\quad s\, Y\, s\, X \mid (\text{other component}).
\]
The total contribution is again
\[
(2x + y + 1) + (x + 2y + 1) = 3x + 3y + 2.
\]

Thus the total contribution from the over-path and under-path
is invariant under moves of type~$\V$.

The above computation also covers the cases where $X$ or $Y$
contains no inter-component crossings, which can be easily  verified
by direct substitution.
\end{proof}
Proposition~\ref{prop:IVdetect} immediately implies Corollary~\ref{cor:pr+1234}.    
\begin{cor}\label{cor:pr+1234}
For any singular link diagram $D$, 
the function $f(\overline{\operatorname{pr}_+}(D))$ is invariant under Reidemeister moves of types~\I, \II, \III, and~\V. 
\end{cor}

\begin{proof}[Proof of Proposition~\ref{4-non-gene}]
We prove the statement by computing the invariant for the diagram
$D$ shown in Figure~\ref{fig:IVOindep}.

\begin{figure}[ht]
    \ExampleO
    \caption{}
    \label{fig:IVOindep}
\end{figure}
The two components appearing in $D$ can be separated by a move of type
$\IVU$ together with ordinary Reidemeister moves of
types~$\I$--$\III$.
After this separation, the resulting diagram has no real crossings;
we denote this diagram by $O$.

Note that the argument does not depend on the particular choice of the
move $\IVU$.
Indeed, any other move of type $\IVU'$ is generated from a given move of type $\IVU$ together with moves of types~$\I$--$\III$ (Proposition~\ref{4-gene}).

On the other hand, we compute
\[
f(\overline{\operatorname{pr}_+}(D))
= (t \cdot t^{-5} + t^5 \cdot t^{-1},\,1)
= (t^{-4} + t^{4},\,1)
\neq (1,1)
= f(O).
\]

The first component of $f(\overline{\operatorname{pr}_+}(D))$
is invariant under Reidemeister moves of types~$\I$--$\III$,
$\IVO$, and $\V$.
Suppose that $\R \cup \mathcal{RIV}_O$
generated some move in $\mathcal{RIV}_U$.
Then, by Proposition~\ref{4-gene}, it would generate
any chosen move $\IVU'$.
In particular, $D$ could be transformed to $O$ using
moves in $\R\cup \mathcal{RIV}_O$.
This would imply
\[
f(\overline{\operatorname{pr}_+}(D)) = f(O),
\]
which is a contradiction.  
Therefore, $\R\cup \mathcal{RIV}_O$
does not generate any move in $\mathcal{RIV}_U$.

By exchanging the roles of $\IVO$ and $\IVU$
and replacing Figure~\ref{fig:IVOindep}
with Figure~\ref{fig:IVUindep},
\begin{figure}
    \centering
\ExampleU
    \caption{}
    \label{fig:IVUindep}
\end{figure}
the same argument shows that
$\R\cup \mathcal{RIV}_U$
does not generate any move in $\mathcal{RIV}_O$.  This completes the proof.

For completeness we record the explicit computation below.
\paragraph{\bf Computation for Figure~\ref{fig:IVOindep}.}

\begin{align*}
\text{Basepoint on the over-path of } & s_1 &
& s_1\, s_1\, c_2\, s_2\, s_2\, c_1 \mid c_1\, c_2 ,\\
\text{Basepoint on the under-path of } & s_1  &
& s_1\, c_2\, s_2\, s_2\, c_1\, s_1 \mid c_1\, c_2 ,\\
\text{Basepoint on the over-path of } & s_2  &
& s_2\, s_2\, c_1\, s_1\, s_1\, c_2 \mid c_1\, c_2 ,\\
\text{Basepoint on the under-path of } & s_2  &
& s_2\, c_1\, s_1\, s_1\, c_2\, s_2 \mid c_1\, c_2 .
\end{align*}

From these words we obtain
\begin{align*}
\tau_o(s_1,c_1) &= 0, &
\tau_o(s_1,c_2) &= 2, \\
\tau_u(s_1,c_1) &= 1, &
\tau_u(s_1,c_2) &= 3, \\
\tau_o(s_2,c_1) &= 2, &
\tau_o(s_2,c_2) &= 0, \\
\tau_u(s_2,c_1) &= 3, &
\tau_u(s_2,c_2) &= 1 .
\end{align*}

Hence
\begin{align*}
\tau(s_1,c_1) &= 1, &
\tau(s_2,c_1) &= 5,\\
\tau(s_1,c_2) &= 5, &
\tau(s_2,c_2) &= 1 .
\end{align*}
\paragraph{\bf Computation for Figure~\ref{fig:IVUindep}.}

In this case the two real crossings are changed by crossing changes.
Therefore the Gauss words above remain the same.
The only difference is that the $\tau$-terms that contributed to $a_+$
now contribute to $a_-$,
and those contributing to $b_-$ now contribute to $b_+$.

Consequently,
\[
f(\overline{\operatorname{pr}_+}(D))
=
(1,\,
t \cdot t^{-5} + t^{5} \cdot t^{-1} )
=
(1,\,
t^{-4} + t^{4})
\neq
(1,1)
=
f(O).
\]
\end{proof} 

\subsection{Detection of type~\V\ moves}\label{sec:Detect-5}

\begin{proof}[Proof of Proposition~\ref{5-non-gene}]
Recall the definition of the labels $a$ and $b$ for inter-component crossings
given in Subsection~\ref{sec:labelAB}.

For an ordered $2$-component link diagram $D$, define
\[
\delta(D)
:=
\#\{\text{inter-component crossings of type }a\}
-
\#\{\text{inter-component crossings of type }b\}.
\]

We claim that $\delta(D)$ is invariant under Reidemeister moves of
types~$\I$--$\IV$.

Indeed, a move of type~$\I$ does not involve any inter-component crossing,
so $\delta(D)$ is unchanged.
Every move of type~$\II$
creates or removes one inter-component crossing of type $a$
and one inter-component crossing of type $b$, and hence their contributions cancel.
Moves of type~$\III$ and $\IV$ do not change the numbers of inter-component crossings
of types $a$ and $b$, so $\delta(D)$ is unchanged.

On the other hand, a move of type~$\V$ is the only Reidemeister move
that changes an inter-component crossing from type $a$ to type $b$,
or from type $b$ to type $a$, without changing the total number of crossings.
Therefore, a move of type~$\V$ changes the value of $\delta(D)$ by $\pm 2$.

Hence any sequence of Reidemeister moves of types~$\I$--$\IV$
preserves $\delta(D)$, whereas a move of type~$\V$ changes it.
Therefore, no move of type~$\V$ can be generated by moves of
types~$\I$--$\IV$.
\end{proof}

\section{Unoriented Case}\label{sec:unori}
We now turn to the unoriented setting.  The following theorem gives the corresponding classification for unoriented singular Reidemeister moves.
\begin{proof}[Proof of Theorem~\ref{thm:unoriented}]
It is known that any generating set of unoriented Reidemeister moves
of types~$\I^{\!\mathrm{un}}$, $\II^{\!\mathrm{un}}$, and $\III^{\!\mathrm{un}}$ must contain at least one move of each type
(cf.~\cite{Ostlund2001PhD, Hagge2006, Manturov2004}).

\medskip
\noindent
\textbf{(Independence of type~$\V^{\!\mathrm{un}}$).}
The move of type~$\V^{\!\mathrm{un}}$ is independent of types~$\I^{\!\mathrm{un}}$--$\IV^{\!\mathrm{un}}$.
Indeed, consider an ordered $2$-component singular link and assign an arbitrary orientation.
Then a move of type~$\V^{\!\mathrm{un}}$ is the only move that changes the crossing type
of an inter-component crossing (see Subsection~\ref{sec:labelAB}),
while moves of types~$\I^{\!\mathrm{un}}$--$\IV^{\!\mathrm{un}}$ preserve these labels.
Hence type~$\V^{\!\mathrm{un}}$ cannot be generated from types~$\I^{\!\mathrm{un}}$--$\IV^{\!\mathrm{un}}$.

\medskip
\noindent
\textbf{(Necessity of both types of  $\IVO^{\mathrm{un}}$ and $\IVU^{\mathrm{un}}$).}
Suppose that no move of type $\IVU^{\mathrm{un}}$ is included.
Given any unoriented $2$-component singular link,
choose an arbitrary orientation.
Then, by the oriented case, the invariant
$f(\overline{\operatorname{pr}_+}(D))$
shows that a move of type $\IVU^{\mathrm{un}}$ is necessary.
Since the choice of orientation is arbitrary,
this argument applies to all possible orientations,
and hence $\IVU^{\mathrm{un}}$ is indispensable in the unoriented setting.
By symmetry, the same holds for $\IVO^{\mathrm{un}}$.

\medskip
\noindent
\textbf{(Sufficiency of one move of each type $\IV^{\!\mathrm{un}}$ and $\V^{\!\mathrm{un}}$).}
Assume that moves of types~$\I^{\!\mathrm{un}}$--$\III^{\!\mathrm{un}}$ are available.
Then the generation results for type~$\IV^{\!\mathrm{un}}$ moves
(Section~\ref{subsec:gene-4})
remain valid in the unoriented setting,
and hence a single move from $\RIVO^{\mathrm{un}}$ (resp.\ $\RIVU^{\mathrm{un}}$)
generates all moves in $\RIVO^{\mathrm{un}}$ (resp.\ $\RIVU^{\mathrm{un}}$).
Similarly, once moves of types~$\I^{\!\mathrm{un}}$--$\IV^{\!\mathrm{un}}$ are available,
the generation results for type~$\V^{\!\mathrm{un}}$ moves
(Section~\ref{subsec:gene-5})
show that a single move of type~$\RV^{\!\mathrm{un}}$ generates all moves in $\RV^{\!\mathrm{un}}$.

\medskip

Combining the above, any minimal generating set must contain
exactly one move from each of
$\RI^{\!\mathrm{un}}$, $\RII^{\!\mathrm{un}}$, $\RIII^{\!\mathrm{un}}$, $\RIVO^{\mathrm{un}}$, $\RIVU^{\mathrm{un}}$, and $\RV^{\!\mathrm{un}}$.
Since there are two choices for each of the three types
$\I^{\!\mathrm{un}}$, $\III^{\!\mathrm{un}}$, and $\V^{\!\mathrm{un}}$,
the total number of such sets is $2\times 1 \times 2 \times 1 \times 1 \times 2=8$.
\end{proof}

\bibliographystyle{plain}
\bibliography{ListSMGS}
\end{document}